\documentclass[final]{siamltex}
\usepackage{amsmath,amsfonts,graphicx}
\newcommand{\Hilb}{\mathbb{H}}
\newcommand{\Poly}{\mathbb{P}}
\newcommand{\R}{\mathbb{R}}
\newcommand{\B}{\mathcal{B}}
\newcommand{\Rev}{\mathcal{R}}
\newcommand{\G}{\mathcal{G}}
\newcommand{\logp}{\ell}
\newcommand{\iprod}[1]{\langle#1\rangle}
\newcommand{\bigiprod}[1]{\bigl\langle#1\bigr\rangle}
\newcommand{\biggiprod}[1]{\bigg\langle#1\bigg\rangle}
\newcommand{\Trial}{\mathcal{W}}

\newcommand{\Id}{\mathcal{I}}
\newcommand{\argmax}{\operatorname{argmax}}
\newcommand{\Lagint}{\mathcal{L}}

\newcommand{\Upp}{U^\sharp}
\newcommand{\sigmapp}{{\sigma^\sharp}}

\title{Superconvergence of a discontinuous Galerkin method for
fractional diffusion and wave equations\thanks{We thank KFUPM for
supporting this research as part of the project SB101020.}}

\author{Kassem Mustapha\thanks{
Department of Mathematics and Statistics,
KFUPM, Dhahran 31261, Saudi Arabia
(\texttt{kassem@kfupm.edu.au}).}
\and William McLean\thanks{
School of Mathematics and Statistics,
The University of New South Wales, Sydney 2052, Australia
(\texttt{w.mclean@unsw.edu.au}).}
}

\begin{document}
\maketitle

\begin{abstract}
We consider an initial-boundary value problem
for~$\partial_tu-\partial_t^{-\alpha}\nabla^2u=f(t)$, that is, for
a fractional diffusion ($-1<\alpha<0$) or wave ($0<\alpha<1$) equation.
A numerical solution is found by applying a piecewise-linear,
discontinuous Galerkin method in time combined with a piecewise-linear,
conforming finite element method in space.  The time mesh is graded
appropriately near~$t=0$, but the spatial mesh is quasiuniform.
Previously, we proved
that the error, measured in the spatial $L_2$-norm, is of
order~$k^{2+\alpha_-}+h^2\ell(k)$, uniformly in~$t$, where $k$ is the maximum
time step, $h$ is the maximum diameter of the spatial finite elements,
$\alpha_-=\min(\alpha,0)\le0$ and $\ell(k)=\max(1,|\log k|)$.
Here, we generalize a known result for
the classical heat equation (i.e., the case~$\alpha=0$) by showing that
at each time level~$t_n$ the solution is superconvergent with
respect to~$k$: the error is of order~$(k^{3+2\alpha_-}+h^2)\ell(k)$.
Moreover, a simple postprocessing step employing Lagrange interpolation
yields a superconvergent approximation for any~$t$.
Numerical experiments indicate that our theoretical error bound is 
pessimistic if~$\alpha<0$. Ignoring logarithmic factors, we observe 
that the error in the DG solution at~$t=t_n$, and after postprocessing 
at all~$t$, is of order~$k^{3+\alpha_-}+h^2$.  
\end{abstract}



\pagestyle{myheadings}
\thispagestyle{plain}
\markboth{KASSEM MUSTAPHA AND WILLIAM MCLEAN}%
{SUPERCONVERGENCE OF A DISCONTINUOUS GALERKIN METHOD}
%
\section{Introduction}\label{intro}
%
In previous work~\cite{McLeanMustapha2009,MustaphaMcLean2009,
MustaphaMcLean2011,MustaphaMcLean2012},
we have studied discontinuous Galerkin (DG) methods for the
time discretization of the abstract intial value problem
\begin{equation}\label{eq: ivp}
u' +\B_\alpha Au  = f(t)\quad \text{
for~$0<t<T$,}\quad\text{with~$u(0)=u_0$},
\end{equation}
where $u'=\partial u/\partial t$ and $\B_\alpha=\partial_t^{-\alpha}$; more 
precisely, letting $\omega_\mu(t)=t^{\mu-1}/\Gamma(\mu)$ for~$\mu>0$, the
function~$\B_\alpha v$ is either a (Riemann--Liouville) fractional order 
derivative in time,
\begin{equation}\label{eq: frac deriv}
\B_\alpha v(t)=\frac{\partial}{\partial t}
    \int_0^t \omega_{1+\alpha}(t-s) v(s)\,ds
    \quad\text{if $-1<\alpha<0$,}
\end{equation}
or a fractional order integral in time,
\[
\B_\alpha v(t)=\int_0^t\omega_\alpha(t-s) v(s)\,ds
    \quad\text{if $0<\alpha<1$.}
\]
In Section~\ref{sec: preliminaries} we set out technical assumptions
on the operator~$A$, but for the present discussion we simply take
$Au=-\nabla^2u$ on a spatial domain~$\Omega$, and impose homogeneous
Dirichlet boundary conditions on~$u$.  

Problems of the form~\eqref{eq: ivp} arise in a variety of physical,
biological and chemical applications~\cite{KilbasSrivastavaTrujillo2006,
Mathai2011,MetzlerKlafter2000,MetzlerKlafter2004,Podlubny1999,SmithEtal1999,
SokolovKlafter2005,Tarasov2011}.  The case~$-1<\alpha<0$ describes slow
or anomalous sub-diffusion and occurs, for example, in models of fractured
or porous media, where the particle flux depends on the entire history
of the density gradient~$\nabla u$.  The case~$0<\alpha<1$ describes
wave propogation in viscoelastic materials \cite{Hanyga2001,
MainardiParadisi2001,Pruss1993}.

In the limit as~$\alpha\to0$, the evolution equation in~\eqref{eq: ivp}
becomes $u'+Au=f$, which is just the classical heat equation, and
Eriksson et al.~\cite{ErikssonJohnsonThomee1985} studied the convergence
of the DG solution~$U(t)\approx u(t)$ in this case.
For a maximum time step~$k$, and using discontinuous piecewise polynomials
of degree at most~$q-1$ in~$t$, with no spatial
discretization, they proved an optimal convergence rate
\[
\|U(t)-u(t)\|\le Ck^q\biggl(\|u_0\|_{q}+\|u^{(q)}(0)\|
    +\|f^{(q-1)}(0)\|+\int_0^{t_n}\|f^{(q)}(s)\|\,ds\biggr),
\]
for $0\le t\le T$, where $\|v\|$ is the norm in~$L_2(\Omega)$ and
$\|v\|_q=\|A^{q/2}v\|$ for $v\in D(A^{q/2})$.
In addition, they proved that the DG solution is superconvergent at the
$n$th time level~$t_n$, satisfying an error bound
\[
\|U(t_n^-)-u(t_n)\|\le Ck^{2q-1}\biggl(\|u_0\|_{2q-1}+\|u^{(q)}(0)\|_{q-1}
    +\int_0^{t_n}\|f^{(q)}(s)\|_{q-1}\,ds\biggr),
\]
where $U(t_n^-)=\lim_{t\to t_n^-}U(t)$ denotes the limit from the left.
Ericksson et al.\ were also able to prove that a
convergence rate faster then~$O(k^q)$ holds under less restrictive
spatial regularity requirements on the solution~$u$.
Our aim is to establish superconvergence results for the fractional-order
problem~\eqref{eq: ivp}, restricting our attention to the
piecewise-\emph{linear} DG method ($q=2$).  We believe
our scheme is the first to achieve better than second-order accuracy
in time.  As well as nodal superconvergence of the DG solution we show
that a postprocessed solution is superconvergent \emph{uniformly} in~$t$.

Many authors have studied numerical methods for~\eqref{eq: ivp}.  In
the case~$0<\alpha<1$, Sanz-Serna~\cite{SanzSerna1988} proposed a 
convolution quadrature scheme, and subsequently Cuesta, Lubich and
Palencia~\cite{CuestaPalencia2003a,CuestaPalencia2003b,
CuestaLubichPalencia2006} developed this approach to obtain 
an~$O(k^2)$ method as well as a fast 
implementation~\cite{SchaedleLopezFernandezLubich2006}.  McLean and
Thom\'ee~\cite{McLeanThomee1993} combined finite differences and
quadrature in time, with finite elements in space. 

In the case~$-1<\alpha<0$, Langlands and Henry~\cite{LanglandsHenry2005}
introduced an implicit Euler scheme involving the 
Gr\"unwald--Letnikov fractional derivative and spatial finite differences
with step size~$h$, and observed $O(k^{1/2}+h^2)$ convergence in
the case~$\alpha=-1/2$.  Yuste and Acedo~\cite{YusteAcedo2005} treated
an explicit Euler scheme and showed $O(k+h^2)$ convergence.  Zhuang, Liu,
Anh, Turner et al.~\cite{ChenLiuAnhTurner2012,LiuYangBurrage2009,
ZhuangLiuAnhTurner2008,ZhuangLiuAnhTurner2009} developed
another class of $O(k+h^2)$ finite difference methods, and 
Yuste~\cite{Yuste2006} presented an $O(k^2+h^2)$ method.
Cui~\cite{Cui2009} and Chen et al.~\cite{ChenLiuAnhTurner2010} 
studied $O(k+h^4)$ schemes, and Cui~\cite{Cui2012a,Cui2012b} analysed an 
$O(k^{\min(1-\alpha,2+\alpha)}+h^4)$ ADI scheme on a rectangular 
spatial domain; see also Wang and Wang~\cite{WangWang2011} and Zhang and 
Sun~\cite{ZhangSun2011}.  For another type of finite difference
scheme~\cite{Mustapha2011,MustaphaMuttawa2012}, the error is
$O(k^{2+\alpha}+h^2)$, and recently Jin et al.~\cite{JinLazarovZhou2012} 
proved optimal error bounds for two semidiscrete finite element methods.
Some of these works employ an alternative 
formulation of~\eqref{eq: ivp} using the Caputo fractional derivative.

In practice, the higher order derivatives of~$u$ are typically
singular~\cite{McLean2010,McLeanMustapha2007} as~$t\to0$, so formally 
high order methods~\cite{ChenLiuAnhTurner2010,ChenLiuAnhTurner2012,
Cui2009,Cui2012a,Cui2012b,LiuYangBurrage2009,ZhangSun2011,
WangWang2011,YusteAcedo2005,ZhuangLiuAnhTurner2008,ZhuangLiuAnhTurner2009} 
can fail to achieve fast 
convergence.  We have analysed several methods that allow for the 
singular behaviour of~$u$ by employing non-uniform time 
steps~\cite{McLeanMustapha2007,McLeanThomeeWahlbin1996,Mustapha2011,%
MustaphaMuttawa2012,MustaphaMcLean2012}.  Another approach, that yields 
a parallel in time algorithm with spectral accuracy 
even for problems with low regularity, is to approximate~$u$ via the 
Laplace inversion formula~\cite{LopezFernandezPalencia2004, 
LopezFernandezPalenciaSchadle2006,McLeanThomee2010}.  

To minimise the need for handling separately the cases $\alpha<0$~and 
$\alpha>0$, it is convenient to write
$\alpha_+=\max(\alpha,0)\ge0$ and $\alpha_-=\min(\alpha,0)\le0$
for the positive and negative parts of~$\alpha$, respectively.
In our theory, we assume that there exist positive constants
$M$~and $\sigma$ such that
\begin{equation}\label{eq: Au'}
\|Au_0\|+\|Au(t)\|\le M\quad\text{and}\quad
\|Au'(t)\|+t\|Au''(t)\|\le Mt^{\sigma-1},
\end{equation}
as well as
\begin{equation}\label{eq: A^2u'}
t\|A^2u'(t)\|+t^2\|A^2u''(t)\|\le Mt^{\sigma-\alpha_--1},
\end{equation}
for $0<t\le T$.  For instance~\cite{McLean2010,McLeanMustapha2007},
if $f\equiv0$ and $u_0\in D(A^2)$, then \eqref{eq:  Au'}~and \eqref{eq: A^2u'}
hold with $M=C\|A^2u_0\|$ and $\sigma=1+\alpha_-$.

Section~\ref{sec: preliminaries} sets out our notation and assumptions,
and recalls some tools and results from earlier 
work~\cite{MustaphaMcLean2011}.  In Section~\ref{sec: dual}, we introduce 
the homogeneous dual problem,
\begin{equation}\label{eq: dual}
-z'+\B_\alpha^*Az=0\quad\text{for $0<t<T$,}\quad\text{with $z(T)=z_T$,}
\end{equation}
for a given terminal value~$z_T$, and represent the nodal
error~$U(t_n^-)-u(t_n)$ in terms of $z(t)$ and its DG
approximation~$Z(t)$.  We allow a
class of non-uniform meshes, specified in Section~\ref{sec: superconv},
where we prove in Theorem~\ref{thm: nodal conv} that the nodal error 
is $O(k^{3+2\alpha_-})$.
Our method of analysis allows us to handle the two cases
$-1<\alpha<0$~and $0<\alpha<1$ together, but the former presents additional
technical difficulties in some places.
In an earlier paper~\cite[Theorem~4.1]{MustaphaMcLean2009}, we estimated
the nodal error for the case $0<\alpha<1$ in a different way that yields a
bound of order~$k^{2+\alpha}$.  (Although we
claimed $O(k^3)$ convergence, the first line
of~\cite[Corollary~4.2]{MustaphaMcLean2009} contains an error.)

In Section~\ref{sec: PP} we construct, via a simple interpolation scheme,
a postprocessed solution~$\Upp$ whose error 
is $O(k^{3+2\alpha_-})$ for \emph{all}~$t$, not just at the nodal values.
Section~\ref{sec: space} introduces a fully discrete scheme by applying
a continuous piecewise-linear, finite element method for the spatial
discretization.  Thus, the fully discrete solution is continuous
in space but discontinuous in time.  We show that the error bound
is as for the semidiscrete method but with an extra term of order~$h^2$.
Finally, we present some numerical examples in Section~\ref{sec: Numerical},
which indicate that our error bounds are pessimistic, at least in some
cases.  We observe that the nodal error from the time discretization is
$O(k^{3+\alpha_-})$, which is better than our theoretical estimate by
a factor~$k^{\alpha_-}$.  The same is true for the postprocessed
solution, uniformly in~$t$.
%
\section{Preliminaries}\label{sec: preliminaries}
%
\subsection{Assumptions on the spatial operator}
We assume as in earlier
work~\cite{ErikssonJohnsonThomee1985,MustaphaMcLean2011}
that the self-adjoint linear operator~$A$ has a complete eigensystem
in a real Hilbert space~$\Hilb$, say $A\phi_j=\lambda_j\phi_j$
for~$j=1$, 2, 3, \dots, and that $A$ is strictly positive-definite with
the eigenvalues ordered so that $0<\lambda_1\le\lambda_2\le\cdots$.
(Strict positive definiteness is not essential, but allowing $\lambda_1=0$
would result in some technical complications that we prefer to avoid.)
We denote the inner product of $u$~and $v$ in~$\Hilb$ by~$\iprod{u,v}$
and the corresponding norm by~$\|u\|=\sqrt{\iprod{u,u}}$.  Associated
with the linear operator~$A$ is a bilinear form, denoted by the
same symbol:
\[
A(u,v)=\sum_{m=1}^\infty\lambda_m\iprod{u,\phi_m}\iprod{\phi_m,v}
    \quad\text{for $u$, $v\in D(A^{1/2})$.}
\]
These assumptions hold, in particular, if $A=-\nabla^2$ subject to
homogenous Dirichlet boundary conditions on a bounded domain~$\Omega$,
because $A$ has a compact inverse on~$\Hilb=L_2(\Omega)$
and~$A(u,v)=\int_\Omega\nabla u\cdot\nabla v\,dx$.

\subsection{The discontinuous Galerkin time discrectization}
Fixing a time interval~$[0,T]$, we introduce a mesh for the time
discretization,
\begin{equation}\label{eq: tn mesh}
0=t_0<t_1<\cdots<t_N=T,
\end{equation}
with $k_n=t_n-t_{n-1}$ and $I_n=(t_{n-1},t_n)$ for $1\le n\le N$,
and a maximum time step~$k=\max_{1\le n\le N}k_n$.  Let $\Poly_r$ denote
the space of polynomials of degree at most~$r$ with coefficients 
in~$D(A^{1/2})$, and let
$J_n=\bigcup_{j=1}^n I_j=[0,t_n]\setminus\{t_0,t_1,\ldots,t_n\}$,
with $J=J_N$.  Our trial space~$\Trial$ consists of the piecewise-linear
functions $U:J\to D(A^{1/2})$ with~$U|_{I_n}\in\Poly_1$ for $1\le n\le N$.
We treat~$U$ as undefined at each time level~$t_n$, and write
\begin{equation}\label{eq: [U]}
U_-^n=U(t_n^-),\quad U_+^n=U(t_n^+),\quad[U]^n=U^n_+-U_-^n.
\end{equation}
For $r\in\{0,1,2,\ldots\}$ we let $C^r(J,\Hilb)$ denote the space of
functions~$v:J\to\Hilb$ such that the restriction~$v|_{I_n}$ extends to an
$r$-times continuously differentiable function on the closed
interval~$[t_{n-1},t_n]$, for~$1\le n\le N$.  In other words,
$v$ is a piecewise~$C^r$ function with respect to the time levels~$t_n$.

If $v\in C^1(J,\Hilb)$, then its fractional derivative~\eqref{eq: frac deriv}
admits the representation~\cite{MustaphaMcLean2011}
\begin{equation}\label{eq: B repn}
\B_\alpha v(t)=\omega_{1+\alpha}(t)v^0_+
    +\sum_{j=1}^{n-1}\omega_{1+\alpha}(t-t_j)[v]^j
    +\sum_{j=1}^n\int_{t_{j-1}}^{\min(t_j,t)}\omega_{1+\alpha}(t-s)
        v'(s)\,ds
\end{equation}
for $t\in I_n$ and $-1<\alpha<0$.
Thus, $\B_\alpha v(t)$ is left-continuous at~$t=t_{n-1}$ but has
a weak singularity
$(t-t_{n-1})^\alpha$ as $t\to t_{n-1}^+$ if~$[v]^{n-1}\ne0$.  However,
if $0<\alpha<1$ then $\B_\alpha v(t)$ is continuous for~$0\le t\le T$.
For $-1<\alpha<1$, the piecewise-linear DG time stepping procedure determines
$U\in\Trial$ by setting $U^0_-=u_0$ and 
requiring~\cite{MustaphaMcLean2009,MustaphaMcLean2011}
\begin{multline}\label{eq: DG step}
\iprod{U^{n-1}_+,X^{n-1}_+}
    +\int_{I_n}
    \bigl[\iprod{U'(t),X(t)}+A\bigl(\B_\alpha U(t),X(t)\bigr)\bigr]\,dt\\
    =\iprod{U_-^{n-1},X^{n-1}_+}+\int_{I_n}\iprod{f(t),X(t)}\,dt,
\end{multline}
for $1\le n\le N$ and for every test function~$X\in\Poly_1$.  The nonlocal
nature of the operator~$\B_\alpha$ means that at each time
step we must compute a sum involving all previous times levels, but this
sum can be evaluated via a fast algorithm~\cite{McLean2012}.
\subsection{Galerkin orthogonality and stability}

For $v\in C^1\bigl(J,D(A^{1/2})\bigr)$ and $w\in C\bigl(J,D(A^{1/2})\bigr)$,
we define the global bilinear form
\begin{equation}\label{eq: GN def}
G_N(v,w)=\iprod{v^0_+,w^0_+}+\sum_{n=1}^{N-1}\iprod{[v]^n,w^n_+}
    +\sum_{n=1}^N\int_{I_n}\bigl[\iprod{v',w}
    +A(\B_\alpha v,w)\bigr]\,dt.
\end{equation}
Summing the equations~\eqref{eq: DG step} gives
\begin{equation}\label{eq: GN U}
G_N(U,X)=\iprod{U_-^0,X^0_+}+\int_0^{t_N}\iprod{f(t),X(t)}\,dt
    \quad\text{for all $X\in\Trial$,}
\end{equation}
and conversely, \eqref{eq: GN U} implies that $U$ satisfies
\eqref{eq: DG step} for~$1\le n\le N$.  Since $[u]^n=0$, 
\begin{equation}\label{eq: GN(u,X)}
G_N(u,X)=\iprod{u_0,X^0_+}+\int_0^{t_N}\iprod{f(t),X(t)}\,dt,
\end{equation}
and thus, assuming $U^0_-=u_0$, the error has the Galerkin orthogonality
property
\begin{equation}\label{eq: Gal orthog}
G_N(U-u,X)=0\quad
\text{for all $X\in\Trial$.}
\end{equation}

The DG method is unconditionally stable.  Indeed, with the notation
\[
\|U\|_I=\sup_{t\in I}\|U(t)\|\quad\text{for any $I\subseteq[0,T]$,}
\]
the following estimate holds.

\begin{theorem}\label{thm: stability}
Given $U_-^0\in\Hilb$ and $f\in L_1\bigl((0,T);\Hilb\bigr)$, there exists
a unique $U\in\Trial$ satisfying \eqref{eq: DG step} for $n=1$, $2$,
\dots, $N$. Furthermore, $U(t)\in D(A)$ for~$t>0$, and
\[
\|U\|_{J_n}^2\le8\biggl|\iprod{U_-^0,U^0_+}
    +\int_0^{t_n}\iprod{f(t),U(t)}\,dt\biggr|
    \quad\text{for $1\le n\le N$.}
\]
\end{theorem}
\begin{proof}
Since $\iprod{V',V}=\frac12(d/dt)\|V\|^2$~and
$\int_0^TA(\B_\alpha V,V)\,dt\ge0$ we find that
\begin{equation}\label{eq: GN(V,V)}
2G_N(V,V)\ge\|V^0_+\|^2+\|V^N_-\|^2+\sum_{n=1}^{N-1}\|[V]^n\|^2
	\quad\text{for all $V\in\Trial$,}
\end{equation}
implying the stated estimate
\cite[Theorem~2.1]{MustaphaMcLean2009},
\cite[Theorem~1]{MustaphaMcLean2011}.
\end{proof}
\subsection{A discontinuous quasi-interpolant}
The conditions
\begin{equation}\label{eq: Pi properties}
\Pi^{-}v(t_n^-)=v(t_n^-)\quad\text{and}\quad
\int_{I_n}\bigl[v(t)-\Pi^- v(t)]\,dt=0
\end{equation}
determine a unique projection operator~$\Pi^-:C(J,\Hilb)\to\Trial$.
Explicitly,
\[
\Pi^{-}v(t):=v(t_n^-)+\frac{v(t_n^-)-\bar v^n}{k_n/2}(t-t_n)
    \quad\text{for $t\in I_n$,}
\]
where $\bar v^n=k_n^{-1}\int_{I_n}v(t)\,dt$ denotes the mean value of~$v$
over~$I_n$, and the interpolation error admits the integral
representations~\cite[Equation~(3.8)]{MustaphaMcLean2009}
\begin{equation}\label{eq: error Pi minus}
\begin{aligned}
\Pi^{-}v(t)-v(t)&=\int_t^{t_n}v'(s)\,ds-2\,\frac{t_n-t}{k_n^2}\int_{I_n}
    (s-t_{n-1})v'(s)\,ds\\
    &=\int_t^{t_n}(t-s)v''(s)\,ds+\frac{t_n-t}{k_n^2}
        \int_{I_n}(s-t_{n-1})^2v''(s)\,ds,
	\quad\text{for $t\in I_n$.}
\end{aligned}
\end{equation}
Likewise, the conditions
\begin{equation}\label{eq: Pi+ properties}
\Pi^{+}v(t_{n-1}^+)=v(t_{n-1}^+)\quad\text{and}\quad
\int_{I_n}\bigl[v(t)-\Pi^+ v(t)]\,dt=0,
\end{equation}
determine a unique projector~$\Pi^+:C(J,\Hilb)\to\Trial$, with
\[
\Pi^{+}v(t):=v(t_{n-1}^+)+\frac{\bar v^n-v(t_{n-1}^+)}{k_n/2}(t-t_{n-1})
    \quad\text{for $t\in I_n$,}
\]
and
\begin{equation}\label{eq: error Pi plus}
\begin{aligned}
\Pi^{+}v(t)-v(t)&=-\int_{t_{n-1}}^tv'(s)\,ds
    +2\,\frac{t-t_{n-1}}{k_n^2}\int_{I_n}
    (t_n-s)v'(s)\,ds\\
    &=\int_{t_{n-1}}^t(s-t_{n-1})v''(s)\,ds+\frac{t-t_{n-1}}{k_n^2}
        \int_{I_n}(t_n-s)^2v''(s)\,ds.
\end{aligned}
\end{equation}
Thus, short calculations lead to the error bound
\begin{equation}\label{eq: ||Pi v-v||}
\|\Pi^\pm v-v\|_{I_n}\le (4-r)k_n^{r-1}\int_{I_n}\|v^{(r)}(t)\|\,dt
    \le (4-r)k_n^r\|v^{(r)}\|_{I_n}
\quad\text{for $r\in\{1,2\}$,}
\end{equation}
and the stability estimates
\begin{equation}\label{eq: Pi norm}
\|\Pi^\pm v\|_{I_n}\le 3\|v\|_{I_n},\quad
\bigl\|(\Pi^\pm)'v\bigr\|_{I_n}\le\frac{2}{k_n}\int_{I_n}\|v'(t)\|\,dt,\quad
\bigl\|[\Pi^\pm v]^n\bigr\|\le\int_{I_n}\|v'(t)\|\,dt.
\end{equation}

%
\section{Dual problem}\label{sec: dual}
%
\subsection{Properties of the adjoint operator}
The adjoint operator appearing in the dual
problem~\eqref{eq: dual} should satisfy, for appropriate $u$~and $v$, 
the identity
\begin{equation}\label{eq: adjoint}
\int_0^T\iprod{v,\B_\alpha w}\,dt
    =\int_0^T\iprod{\B_\alpha^*v, w}\,dt,
\end{equation}
and the next lemma establishes an explicit representation of~$\B_\alpha^*$.

\begin{lemma}\label{lem: B repn}
The identity~\eqref{eq: adjoint} holds in the following cases.
\begin{enumerate}
\item
If $-1<\alpha<0$ and $v$, $w\in C^1(J,\Hilb)$, with
\[
\B_\alpha^*w(t)=-\frac{\partial}{\partial t}\int_t^T
    \omega_{1+\alpha}(s-t)w(s)\,ds
    \quad\text{for $t\in J$.}
\]
\item
If $0<\alpha<1$ and $v$, $w\in C(J,\Hilb)$, with
\[
\B_\alpha^*w(t)=\int_t^T \omega_\alpha(s-t)w(s)\,ds\quad
    \text{for $0\le t\le T$.}
\]
\end{enumerate}
\end{lemma}
\begin{proof}
In case~1, we see from the representation~\eqref{eq: B repn} that
\[
\int_0^T\iprod{\B_\alpha v,w}\,dt=\sum_{n=1}^N\int_{I_n}
    \iprod{\B_\alpha v,w}\,dt=S_1+S_2+S_3,
\]
where, letting $B^{n,j}=\int_{I_n}\omega_{1+\alpha}(t-t_j)w(t)\,dt$ and
\[
D_{nj}=\int_{I_n}\int_{t_{j-1}}^{\min(t_j,t)} 
	\omega_{1+\alpha}(t-s)\iprod{v'(s),w(t)}\,ds\,dt,
\]
we define
\[
S_1=\sum_{n=1}^N\bigiprod{v^0_+,B^{n,0}},\quad
S_2=\sum_{n=2}^N\sum_{j=1}^{n-1}\bigiprod{[v]^j,B^{n,j}},\quad
S_3=\sum_{n=1}^N\sum_{j=1}^n D_{nj}.
\]
By reversing the order of integration, integrating by parts and then
interchanging variables, we find that for $1\le j\le n-1$,
\[
D_{nj}=\iprod{v^j_-,B^{n,j}}-\iprod{v^{j-1}_+,B^{n,j-1}}
    -\int_{I_j}\biggiprod{v(t),\frac{\partial}{\partial t}\int_{I_n}
        \omega_{1+\alpha}(s-t)w(s)\,ds}\,dt,
\]
whereas 
$D_{nn}=-\iprod{v^{n-1}_+,B^{n,n-1}}
    -\int_{I_n}\bigiprod{v(t),\frac{\partial}{\partial t}\int_t^{t_n}
        \omega_{1+\alpha}(s-t)w(s)\,ds}\,dt$.
Thus, after interchanging the order of summation for the double integrals,
\[
S_3=-\sum_{n=1}^N\iprod{v^{n-1}_+,B^{n,n-1}}+\sum_{n=2}^N\sum_{j=1}^{n-1}
    \bigl(\iprod{v^j_-,B^{n,j}}-\iprod{v^{j-1}_+,B^{n,j-1}}\bigr)
    -\sum_{j=1}^N\int_{I_j}\iprod{v,\B^*_\alpha w}\,dt,
\]
that is,
\begin{align*}
S_3&=\sum_{n=2}^N\sum_{j=1}^{n-1}\iprod{v^j_-,B^{n,j}}
     -\iprod{v^0_+,B^{1,0}}
     -\sum_{n=2}^N\sum_{j=0}^{n-1}\iprod{v^j_+,B^{n,j}}
     -\int_0^T\iprod{v,\B^*_\alpha w}\,dt\\
    &=-S_1-S_2-\int_0^T\iprod{v,\B^*_\alpha w}\,dt,
\end{align*}
so \eqref{eq: adjoint} holds.  In the case $0<\alpha<1$, we simply reverse 
the order of integration.
\end{proof}

The adjoint operator admits a representation analogous to~\eqref{eq: B repn}.

\begin{lemma}
If $-1<\alpha<0$, then $\B^*_\alpha w(t)$ equals
\[
\omega_{1+\alpha}(t_n-t)w^N_-
    -\sum_{j=n}^{N-1}\omega_{1+\alpha}(t_j-t)[w]^j
    -\sum_{j=n}^N\int_{\max(t_{j-1},t)}^{t_j}
    \omega_{1+\alpha}(s-t)w'(s)\,ds
\]
for $w\in C^1(J,\Hilb)$ and $t\in I_n$.
Thus, $\B^*_\alpha w(t)$ is right-continuous at~$t=t_n$ but possesses a
weak singularity~$(t_n-t)^\alpha$ as $t\to t_n^-$.
\end{lemma}

\begin{proof}
If $t\in I_n$ and $j\ge n+1$, then, integrating by parts,
\[
\int_{I_j}\omega_{1+\alpha}(s-t)w(s)\,ds
    =\omega_{2+\alpha}(t_j-t)w^j_--\omega_{2+\alpha}(t_{j-1}-t)w^{j-1}_+
    -\int_{I_j}\omega_{2+\alpha}(s-t)w'(s)\,ds
\]
and $\int_t^{t_n}\omega_{1+\alpha}(s-t)w(s)\,ds=\omega_{2+\alpha}(t_n-t)w^n_-
-\int_t^{t_n}\omega_{2+\alpha}(s-t)w'(s)\,ds$.
Differentiating these expressions with respect to~$t$, we see from
part~1 of Lemma~\ref{lem: B repn} that $\B_\alpha^*w(t)$ equals
\[
\sum_{j=n}^N\omega_{1+\alpha}(t_j-t)w^j_-
    -\sum_{j=n+1}^N\omega_{1+\alpha}(t_{j-1}-t)w^{j-1}_+
    -\sum_{j=n}^N\int_{\max(t_{j-1},t)}^{t_j}
    \omega_{1+\alpha}(s-t)w'(s)\,ds,
\]
and the result follows after shifting the index in the second sum.
\end{proof}

\subsection{Representation of the nodal error}
Integration by parts in~\eqref{eq: GN def}, together with the
identity~\eqref{eq: adjoint}, shows that for all $v$, $w\in C^1(J,\Hilb)$,
\begin{equation}\label{eq: GN dual}
G_N(v,w)=\iprod{v^N_-,w^N_-}-\sum_{n=1}^{N-1}\iprod{v^n_-,[w]^n}\\
    +\sum_{n=1}^N\int_{I_n}\bigl[-\iprod{v,w'}
        +A(v,\B_\alpha^*w)\bigr]\,dt.
\end{equation}
Since $-\iprod{v,z'}+A(v,\B_\alpha^*z)=\iprod{v,-z'+\B_\alpha^*Az}=0$,
the solution~$z$ of the dual problem~\eqref{eq: dual} satisfies
\begin{equation}\label{eq: dual weak}
G_N(v,z)=\iprod{v^N_-,z_T}
\quad\text{for all $v\in C\bigl(J,D(A^{1/2})\bigr)$.}
\end{equation}
We therefore define the DG solution~$Z\in\Trial$ of~\eqref{eq: dual} by
\begin{equation}\label{eq: dual DG}
G_N(V,Z)=\iprod{V^N_-,Z^N_+}\quad\text{for all $V\in\Trial$,}
\end{equation}
with $Z^N_+=z_T$, and deduce the Galerkin orthogonality property
\begin{equation}\label{eq: dual Gal orthog}
G_N(V,Z-z)=0\quad\text{for all $V\in\Trial$.}
\end{equation}
The following representation is the basis for our analysis of the nodal
error.

\begin{theorem}\label{thm: nodal error}
If $u$~and $z$ are the solutions of the initial-value
problem~\eqref{eq: ivp} and of the dual problem~\eqref{eq: dual},
and if $U$~and $Z$ are the corresponding DG solutions, then
\[
\iprod{U^N_--u(t_N),z_T}=G_N(u-\Pi^-u,Z-z)\quad\text{for every $z_T\in\Hilb$.}
\]
\end{theorem}
\begin{proof}
Taking $V=U$ in~\eqref{eq: dual DG} and $v=u$ in~\eqref{eq: dual weak} gives
\[
\iprod{U^N_--u(t_N),z_T}=\iprod{U^N_-,z_T}-\iprod{u(t_N),z_T}
    =G_N(U,Z)-G_N(u,z)=G_N(u,Z-z),
\]
where the last step used the Galerkin orthogonality
property~\eqref{eq: Gal orthog} of~$U$, with~$X=Z$.  
Now use the Galerkin orthogonality
property~\eqref{eq: dual Gal orthog} of~$Z$, with $V=\Pi^-u$.
\end{proof}

\subsection{Error in the DG solution of the dual problem}
We will use the following regularity estimates.

\begin{lemma}\label{lem: dual reg}
For $-1<\alpha<1$ and $0<t<T$, the solution~$z$ of the dual
problem~\eqref{eq: dual} satisfies
\[
\|A^{-1}z'(t)\|+(T-t)\|A^{-1}z''(t)\|\le C(T-t)^\alpha\|z_T\|
\]
and
\[
(T-t)^{1+\alpha}\|Az(t)\|+\|z(t)\|+(T-t)\|z'(t)\|\le C\|z_T\|.
\]
\end{lemma}
\begin{proof}
Define the time reversal operator~$\Rev v(t)=v(T-t)$.  Since 
$\Rev\partial_t=-\partial_t\Rev$
and $\Rev\B_\alpha^*=\B_\alpha\Rev$, we deduce from~\eqref{eq: dual} that
the function~$v=\Rev A^{-1}z$ satisfies
\[
v'+\B_\alpha Av=0\quad\text{for $0<t<T$,}\quad
    \text{with $v(0)=A^{-1}z_T$.}
\]
Known results for $-1<\alpha<0$ \cite[Theorem~4.2]{McLean2010}
and $0<\alpha<1$ \cite[Theorem~2.1]{McLeanMustapha2007} give
$\|v'(t)\|+t\|v''(t)\|\le Ct^\alpha\|Av(0)\|
=Ct^\alpha\|z_T\|$, implying the first estimate.  
Similarly \cite[Theorem~4.1]{McLean2010}, the function~$w=\Rev z$ satisfies
\[
t^{1+\alpha}\|Aw(t)\|+\|w(t)\|+t\|w'(t)\|\le C\|w(0)\|=C\|z_T\|,
\]
implying the second estimate.
\end{proof}

To investigate the DG error for the dual problem, we make the splitting
\begin{equation}\label{eq: split Z-z}
A^{-1}(Z-z)=\zeta+\Theta\quad
\text{where $\zeta=A^{-1}(\Pi^+z-z)$ and $\Theta=A^{-1}(Z-\Pi^+z)\in\Trial$.}
\end{equation}

\begin{lemma}\label{lem: zeta}
The function~$\zeta$ in~\eqref{eq: split Z-z} satisfies
$\|\zeta\|_J\le Ct_N^{\alpha_+}k^{1+\alpha_-}\|z_T\|$.
\end{lemma}
\begin{proof}
By~\eqref{eq: ||Pi v-v||} and Lemma~\ref{lem: dual reg},
$\|\zeta\|_J$ is bounded by
\[
\|(\Id-\Pi^+)A^{-1}z\|_J
    \le3\max_{1\le n\le N}\int_{I_n}\|A^{-1}z'(t)\|\,dt
    \le C\|z_T\|\max_{1\le n\le N}\int_{I_n}(t_N-t)^\alpha\,dt.
\]
If $-1<\alpha<0$, then $(1+\alpha)\int_{I_n}(t_N-t)^\alpha\,dt
=(t_N-t_{n-1})^{1+\alpha}-(t_N-t_n)^{1+\alpha}\le k_n^{1+\alpha}$,
whereas if $0<\alpha<1$, then $\int_{I_n}(t_N-t)^\alpha\,dt\le k_nt_N^\alpha$.
\end{proof}

\begin{lemma}\label{lem: Theta}
The function~$\Theta$ in~\eqref{eq: split Z-z} satisfies
$\|\Theta\|_J^2\le8\bigl|\int_0^T
\iprod{\Theta(t),\B_\alpha^*A\zeta}\,dt\bigr|$.
\end{lemma}
\begin{proof}
By~\eqref{eq: dual Gal orthog}, $G_N(V,\zeta+\Theta)=G(A^{-1}V,Z-z)=0$
for all $V\in\Trial$,
where we used the identity $G_N(v,A^{-1}w)=G_N(A^{-1}v,w)$ and the fact that
$A^{-1}V\in\Trial$.  Thus,
\[
G_N(V,\Theta)=-G(V,\zeta)\quad\text{for all $V\in\Trial$.}
\]
Since $\zeta^n_+=0$ for~$0\le n\le N-1$, the formula~\eqref{eq: GN dual}
shows
\[
G_N(V,\zeta)=\sum_{n=1}^N\iprod{V^n_-,\zeta^n_-}+\sum_{n=1}^N
    \int_{I_n}\bigl[-\iprod{V,\zeta'}
    +A(V,\B_\alpha^*\zeta)\bigr]\,dt,
\]
and integration by parts gives $\int_{I_n}\iprod{V,\zeta'}\,dt
=\iprod{V^n_-,\zeta^n_-}-\int_{I_n}\iprod{V',\zeta}\,dt=\iprod{V^n_-,\zeta^n_-}$,
where, in the last step, we used the second property
in~\eqref{eq: Pi+ properties} and the fact that $V'$ is constant on~$I_n$.
Thus, if we define $g=-A\B_\alpha^*\zeta=-\B_\alpha^*A\zeta$ then
\[
G_N(V,\Theta)=-\int_0^TA(V,\B_\alpha^*\zeta)\,dt
    =\int_0^T\iprod{V,g}\,dt
\quad\text{for all $V\in\Trial$,}
\]
which means that $\Theta\in\Trial$ is the DG solution of
$-\theta'+\B_\alpha^*A\theta=g(t)$ for $0<t<T$, with~$\theta(T)=0$.
The desired estimate follows by the stability of~$\Theta$, which we can
prove by applying Theorem~\ref{thm: stability} to~$\Rev\Theta(t)=\Theta(T-t)$.
\end{proof}

Recall that $\logp(t)=\max(1,|\log t|)$.

\begin{lemma}\label{lem: bilinear}
If $-1<\alpha<1$ then
\[
\biggl|\int_0^T\iprod{V,\B_\alpha^*A\zeta}\,dt\biggr|
    \le Ct_N^{\alpha_+}k^{1+\alpha_-}\logp(t_N/k_N)\|V\|_J\|z_T\|
    \quad\text{for all $V\in\Trial$.}
\]
\end{lemma}
\emph{Proof}.
Suppose first that $-1<\alpha<0$.
Since $\B_\alpha V=(\B_{1+\alpha}V)'$ and $\zeta^{n-1}_+=0$, we see
using \eqref{eq: adjoint} and integrating by parts that
\[
\int_0^T\iprod{V,\B^*_\alpha A\zeta}\,dt=\sum_{n=1}^N\int_{I_n}
    \iprod{(\B_{1+\alpha}V)',A\zeta}\,dt
    =\sum_{n=1}^N\int_{I_n}\iprod{\Delta^n,A\zeta'}\,dt,
\]
where, for $t\in I_n$,
\begin{align*}
\Delta^n(t)&=\B_{1+\alpha}V(t_n)-\B_{1+\alpha}V(t)\\
    &=\int_0^t[\omega_{1+\alpha}(t_n-s)-\omega_{1+\alpha}(t-s)]V(s)\,ds
    +\int_t^{t_n}\omega_{1+\alpha}(t_n-s)V(s)\,ds.
\end{align*}
The function~$\omega_{1+\alpha}$ is monotone decreasing whereas
$\omega_{2+\alpha}$ is monotone increasing, so
\begin{align*}
\|\Delta^n(t)\|&\le\|V\|_J\biggl(\int_0^t
    \bigl|\omega_{1+\alpha}(t_n-s)-\omega_{1+\alpha}(t-s)\bigr|\,ds
    +\int_t^{t_n}\omega_{1+\alpha}(t_n-s)\,ds\biggr)\\
    &=\|V\|_J\bigl[\omega_{2+\alpha}(t)-\omega_{2+\alpha}(t_n)
        +2\omega_{2+\alpha}(t_n-t)\bigr]
    \le2\|V\|_J\omega_{2+\alpha}(t_n-t)
\end{align*}
and the Cauchy--Schwarz inequality shows that
$\bigl|\int_0^T\iprod{V,\B_\alpha^*A\zeta}\,dt\bigr|$ is bounded by
\[
2\|V\|_J\biggl(\sum_{n=1}^{N-1}\omega_{2+\alpha}(k_n)
    \int_{I_n}\|A\zeta'\|\,dt
    +\int_{I_N}\omega_{2+\alpha}(t_N-t)\|A\zeta'\|\,dt\biggr).
\]
The integral representation of the interpolation
error~\eqref{eq: error Pi plus} and Lemma~\ref{lem: dual reg} imply
\begin{align*}
\sum_{n=1}^{N-1}\omega_{2+\alpha}(k_n)\int_{I_n}\|A\zeta'\|\,dt
    &\le C\sum_{n=1}^{N-1}k_n^{1+\alpha}\int_{I_n}\|z'\|\,dt\\
    &\le Ck^{1+\alpha}\|z_T\|\int_0^{T-k_N}(T-t)^{-1}\,dt
    =Ck^{1+\alpha}\log\frac{t_N}{k_N}
\end{align*}
and $\int_{I_N}\omega_{2+\alpha}(t_N-t)\|A\zeta'\|\,dt
\le C\|z_T\|\int_{I_N}(t_N-t)^\alpha\,dt\le C\|z_T\|k_N^{1+\alpha}$.
The desired estimate follows at once.

Now let $0<\alpha<1$.
By part~2 of Lemma~\ref{lem: B repn},
\begin{align*}
\biggl|\int_0^T&\iprod{V,\B_\alpha^*A\zeta}\,dt\biggr|
    \le\int_0^T\|V(t)\|\int_t^T\omega_\alpha(s-t)\|A\zeta(s)\|\,ds\,dt\\
    &\le\|V\|_J\int_0^T\|A\zeta(s)\|\int_0^s\omega_\alpha(s-t)\,dt\,ds
    =\|V\|_J\int_0^T\|A\zeta(s)\|\omega_{1+\alpha}(s)\,ds\\
    &\le CT^\alpha\|V\|_J\int_0^T\|A\zeta(t)\|\,dt.
\end{align*}
The estimates \eqref{eq: ||Pi v-v||}~and \eqref{eq: Pi norm} imply that
\[
\int_0^T\|A\zeta(t)\|\,dt\le\sum_{n=1}^N k_n\|A\zeta\|_{I_n}
    \le4k_N\|z\|_{I_N}+3\sum_{n=1}^{N-1}k_n\int_{I_n}\|z'(t)\|\,dt,
\]
and we know from Lemma~\ref{lem: dual reg} that $\|z\|_{I_N}\le C\|z_T\|$ and
\[
\sum_{n=1}^{N-1}k_n\int_{I_n}\|z'(t)\|\,dt
    \le Ck_n\|z_T\|\int_0^{t_{N-1}}(t_N-t)^{-1}\,dt
    =Ck_N\|z_T\|\log(t_N/k_N).\eqno\endproof
\]

Hence, we arrive at the following error estimate for the dual problem.

\begin{theorem}\label{thm: weak ||Z-z||}
Let $z$ denote the solution of the dual problem~\eqref{eq: dual}, and let
$Z$ denote the DG solution defined by~\eqref{eq: dual DG}.  Then,
for $-1<\alpha<1$,
\[
\|A^{-1}(Z-z)\|_J\le Ct_N^{\alpha_+}k^{1+\alpha_-}\logp(t_N/k_N)\|z_T\|.
\]
\end{theorem}
\begin{proof}
The splitting~\eqref{eq: split Z-z} implies that
$\|A^{-1}(Z-z)\|_J\le\|\zeta\|_J+\|\Theta\|_J$, and we estimate these two
terms using Lemmas \ref{lem: zeta}, \ref{lem: Theta}~and
\ref{lem: bilinear}.
\end{proof}
%
\section{Nodal superconvergence}\label{sec: superconv}
%
With the help of Theorems \ref{thm: nodal error}~and
\ref{thm: weak ||Z-z||},
we are now able to estimate the error in the
approximation~$U^n_-\approx u(t_n)$.  Define
\begin{equation}\label{eq: E}
\epsilon(u)=D_1+E_1+\max_{2\le j\le n}k_jD_j
    +\sum_{j=2}^nk_j^{2+\alpha_-}E_j,
\end{equation}
where
\begin{equation}\label{eq: D1 E1}
D_1=\int_{I_1}\|Au'(t)\|\,dt
\quad\text{and}\quad
E_1=\int_{I_1}t^{1+\alpha_-}\|A^2u'(t)\|\,dt,
\end{equation}
with
\begin{equation}\label{eq: Dn En}
D_n=\int_{I_n}\|Au''(t)\|\,dt
\quad\text{and}\quad
E_n=\int_{I_n}\|A^2u''(t)\|\,dt
\quad\text{for $2\le n\le N$.}
\end{equation}

\begin{theorem}\label{thm: ||nodal error||}
Let $u$ be the solution of the initial value problem~\eqref{eq: ivp} and
let $U$ be the DG solution satisfying~\eqref{eq: DG step}.  Then,
for $1\le n\le N$,
\begin{equation}\label{eq: nodal error}
\|U^n_--u(t_n)\|\le Ct_n^{2\alpha_+}k^{1+\alpha_-}\logp(t_n/k_n)\,
    \epsilon(u).
\end{equation}
\end{theorem}
\begin{proof}
Put $\eta=u-\Pi^-u$ and define
\[
\delta^n_1=\int_{I_n}\iprod{\eta,(z-Z)'}\,dt
\quad\text{and}\quad
\delta^n_2=\int_{I_n}\iprod{\B_\alpha A\eta,z-Z}\,dt.
\]
Since $\eta^n_-=0$ for~$1\le n\le N$, we see from
Theorem~\ref{thm: nodal error}, Lemma~\ref{lem: B repn} and
\eqref{eq: GN dual} that
\[
\iprod{U^N_--u(t_n),z_T}=G_N(\eta,Z-z)
    =\sum_{n=1}^N\bigl(\delta^n_1+\delta^n_2\bigr).
\]
Since $Z'$ is constant on~$I_n$, the second property of~$\Pi^-$
in~\eqref{eq: Pi properties} gives
\[
\delta^n_1=\int_{I_n}\iprod{\eta,z'}\,dt
    =\int_{I_n}\bigiprod{\eta(t),z'(t)-z'(t_{n-1})}\,dt
    =\int_{I_n}\int_{t_{n-1}}^t\bigiprod{A\eta(t),A^{-1}z''(s)}\,ds\,dt,
\]
and therefore, using Lemma~\ref{lem: dual reg},
\[
\sum_{n=1}^{N-1}|\delta^n_1|\le\|A\eta\|_{J}\sum_{n=1}^{N-1}
    k_n\int_{I_n}\|A^{-1}z''(t)\|\,dt
    \le C\|A\eta\|_{J_N}\|z_T\|\sum_{n=1}^{N-1}
        k_n\int_{I_n}(t_N-t)^{\alpha-1}\,dt,
\]
whereas $|\delta^N_1|=\bigl|\int_{I_N}\iprod{A\eta,A^{-1}z'}\,dt\bigr|
\le C\|A\eta\|_{I_N}\|z_T\|\int_{I_N}(t_N-t)^\alpha\,dt$.
Here,
\begin{multline*}
\sum_{n=1}^{N-1} k_n\int_{I_n}(t_N-t)^{\alpha-1}\,dt
    \le k\int_0^{t_{N-1}}(t_N-t)^{\alpha-1}\,dt
    =\frac{k}{\alpha}(t_N^\alpha-k_N^\alpha)
    \le C t_N^{\alpha_+}k^{1+\alpha_-},
\end{multline*}
and likewise
$\int_{I_N}(t_N-t)^\alpha\,dt=k_N^{1+\alpha}/(1+\alpha)
\le Ct_N^{\alpha_+}k^{1+\alpha_-}$.  By~\eqref{eq: ||Pi v-v||},
$\|A\eta\|_{I_1}\le3D_1$ and
$\|A\eta\|_{I_n}\le2k_nD_n$ for $2\le n\le N$, so
\begin{equation}\label{eq: delta 1 bound}
\sum_{n=1}^N|\delta^n_1|\le Ct_N^{2\alpha_+}k^{1+\alpha_-}
    \|z_T\|\biggl(D_1+\max_{2\le n\le N}k_nD_n\biggr)
    \quad\text{for $-1<\alpha<1$.}
\end{equation}

Turning to~$\delta^n_2$, if $-1<\alpha<0$,
then \cite[Lemma~2]{MustaphaMcLean2011}
\[
\biggl|\sum_{n=1}^N\delta^n_2\biggr|=\biggl|\int_0^{t_N}
    \iprod{\B_\alpha A^2\eta,A^{-1}(Z-z)}\,dt\biggr|
    \le C\|A^{-1}(Z-z)\|_J\biggl(E_1+\sum_{n=2}^N k_n^{2+\alpha}E_n\biggr),
\]
but if $0<\alpha<1$ then $|\delta^n_2|$ is bounded by
\[
\|A^{-1}(Z-z)\|_{I_n}\int_{I_n}\|\B_\alpha A^2\eta(t)\|\,dt
    \le\|A^{-1}(Z-z)\|_{I_n}\int_{I_n}\int_0^t
        \omega_\alpha(t-s)\|A^2\eta(s)\|\,ds,
\]
so, after summing over~$n$ and reversing the order of integration,
\[
\sum_{n=1}^N|\delta^n_2|
    \le Ct_N^\alpha\|A^{-1}(Z-z)\|_J\int_0^{t_N}\|A^2\eta(t)\|\,dt.
\]
The integral representation~\eqref{eq: error Pi minus} implies that
\begin{align*}
\int_{I_1}\|A^2\eta(t)\|&\le\int_{I_1}\biggl(\int_t^{t_1}\|A^2u'(s)\|\,ds
    +2\frac{t_1-t}{k_1^2}\int_{I_1}s\|A^2u'(s)\|\,ds\biggr)\,dt\\
    &=\int_{I_1}\|A^2u'(s)\|\biggl(\int_0^s\,dt+\frac{s}{k_1^2}
        \int_{I_1}2(t_1-t)\,dt\biggr)\,ds=2E_1,
\end{align*}
and by~\eqref{eq: ||Pi v-v||},
$\int_{t_1}^{t_N}\|A^2\eta(t)\|\,dt\le\sum_{n=2}^Nk_n\|A^2\eta\|_{I_n}
\le\sum_{n=2}^N2k_n^2 E_j$.
Applying Theorem~\ref{thm: weak ||Z-z||},
\begin{equation}\label{eq: delta 2 bound}
\sum_{n=1}^N|\delta^n_2|\le Ct_N^{2\alpha_+}k^{1+\alpha_-}
    \ell(t_N/k_N)\|z_T\|
    \biggl(E_1+\sum_{n=2}^Nk_n^{2+\alpha_-}E_n\biggr)
    \quad\text{for $-1<\alpha<1$.}
\end{equation}
Since $z_T\in\Hilb$ is arbitrary, the desired estimate follows from
\eqref{eq: delta 1 bound}~and \eqref{eq: delta 2 bound}.
\end{proof}

To estimate the convergence rate at the nodes, we introduce some assumptions
about the behaviour of the time steps, namely that, for some fixed~$\gamma\ge1$,
\begin{equation}\label{eq: kn}
k_n\le C_\gamma k\min(1, t_n^{1-1/\gamma})
\quad\text{and}\quad
t_n\le C_\gamma t_{n-1}\quad\text{for $2\le n\le N$,}
\end{equation}
with
\begin{equation}\label{eq: k1}
c_\gamma k^\gamma\le k_1\le C_\gamma k^{\gamma}.
\end{equation}
For example, these assumptions are satisfied if we put
\begin{equation}\label{eq: std tn}
t_n=(n/N)^\gamma T\quad\text{for $0\le n\le N$.}
\end{equation}

\begin{lemma}\label{lem: nodal conv}
Assume that $u$ satisfies \eqref{eq: Au'}~and \eqref{eq: A^2u'}, and that
the time mesh satisfies \eqref{eq: kn}~and \eqref{eq: k1}.  Then, with
$\gamma^*=(2+\alpha_-)/\sigma$ and for $1\le n\le N$,
\[
\epsilon(u)\le C_TM\times\begin{cases}
k^{\gamma\sigma},&1\le\gamma<\gamma^*,\\
k^{2+\alpha_-}\logp(t_n/k_1),&\gamma=\gamma^*,\\
k^{2+\alpha_-},&\gamma>\gamma^*.
\end{cases}
\]
\end{lemma}
\emph{Proof}.
The stated assumptions imply that
$
D_1+E_1\le CM\int_0^{k_1}t^{\sigma-1}\,dt\le CMk_1^\sigma
    \le CMk^{\gamma\sigma},
$
and, for $2\le j\le n$,
\begin{multline*}
k_jD_j\le CMk_j\int_{I_j}t^{\sigma-2}\,dt\le CMk_j^2t_j^{\sigma-2}
    \le CM\times\begin{cases}
    k^{\gamma\sigma},&1\le\gamma<2/\sigma,\\
    t_n^{\sigma-2/\gamma}k^2,&\gamma\ge2/\sigma.
    \end{cases}\\
    \le CM\times\begin{cases}
    k^{\gamma\sigma},&1\le\gamma\le\gamma^*,\\
    k^{2+\alpha_-},&\gamma^*\le\gamma\le2/\sigma,\\
    t_n^{\sigma-(2+\alpha_-)/\gamma}k^{2+\alpha_-},&\gamma\ge\gamma^*.
    \end{cases}
\end{multline*}
Similarly,
\begin{align*}
\sum_{j=2}^n k_j^{2+\alpha_-}E_j
    &\le M\sum_{j=2}^n k_j^{2+\alpha_-}\int_{I_j}t^{\sigma-3-\alpha_-}\,dt
    \le CMk^{2+\alpha_-}\int_{k_1}^{t_n}
        t^{\sigma-1-(2+\alpha_-)/\gamma}\,dt\\
    &\le CMk^{2+\alpha_-}\times\begin{cases}
    k^{\gamma\sigma-(2+\alpha_-)},&1<\gamma<\gamma^*,\\
    \log(t_n/k_1),                   &\gamma=\gamma^*,\\
    t_n^{\sigma-(2+\alpha_-)/\gamma},&\gamma>\gamma^*.\qquad\endproof
    \end{cases}
\end{align*}

We can now state our main result on nodal superconvergence.

\begin{theorem}\label{thm: nodal conv}
Assume that the solution~$u$ of the initial value problem~\eqref{eq: ivp} 
satisfies \eqref{eq: Au'}~and \eqref{eq: A^2u'}, and that
the time mesh satisfies \eqref{eq: kn}~and \eqref{eq: k1} 
with~$\gamma>\gamma^*=(2+\alpha_-)/\sigma$.  Then, 
for the DG method~\eqref{eq: DG step}, we have the error bound
\[
\|U^n_--u(t_n)\|\le CMt_n^{2\alpha_+}k^{3+2\alpha_-}\logp(t_n/k_n)
	\quad\text{for $1\le n\le N$.}
\]
\end{theorem}
\begin{proof}
The error bound follows at once from Theorem
\ref{thm: ||nodal error||}~and Lemma~\ref{lem: nodal conv}.
\end{proof}
%
\section{Postprocessing}\label{sec: PP}
%
We can postprocess the DG solution~$U$ to obtain a globally
superconvergent solution~$\Upp$ using simple Lagrange interpolation,
as follows.  Given a piecewise continuous function~$v:J\to\Hilb$, define
$\Lagint v:J\to\Hilb$ by linear interpolation on the first two subintervals,
\begin{equation}\label{eq: L I1}
\Lagint v(t)=k_n^{-1}[(t_n-t)v^{n-1}_-+(t-t_{n-1})v^n_-]
	\quad\text{for $t\in I_n$ and $n\in\{1,2\}$,}
\end{equation}
and backward quadratic interpolation on the remaining subintervals,
\begin{equation}\label{eq: L In}
\Lagint v(t)=\frac{(t-t_{n-1})(t-t_n)}{k_{n-1}(k_{n-1}+k_n)}\,v^{n-2}_-
    -\frac{(t-t_{n-2})(t-t_n)}{k_{n-1}k_n}\,v^{n-1}_-
    +\frac{(t-t_{n-2})(t-t_{n-1})}{(k_{n-1}+k_n)k_n}\,v^n_-
\end{equation}
for $t\in I_n$ and $n\ge3$.  Thus, $(\Lagint v)(t_n)=v^n_-$
for~$0\le n\le N$, and we define the postprocessed solution by
\begin{equation}\label{eq: U*}
\Upp=\Lagint U.
\end{equation}
The interpolant of the exact solution satisfies the following error bound.

\begin{lemma}\label{lem: L error}
If there exist positive constants $M$~and $\sigmapp$ such that
\begin{equation}\label{eq: u'''}
\|u'(t)\|+t^2\|u'''(t)\|\le Mt^{\sigmapp-1}
    \quad\text{for $0<t\le T$,}
\end{equation}
and if the time mesh satisfies \eqref{eq: kn}~and \eqref{eq: k1}
with~$\gamma\ge3/\sigmapp$, then
\[
\|u-\Lagint u\|_J\le CM\times\begin{cases}
k^{\gamma\sigmapp},&1\le\gamma<3/\sigmapp,\\
T^{\sigmapp-3/\gamma}k^3,&\gamma\ge3/\sigmapp.
\end{cases}
\]
\end{lemma}
\emph{Proof}.
If $n\in\{1,2\}$ and $t\in I_n$, then
\[
(u-\Lagint u)(t)=\int_{t_{n-1}}^t u'(s)\,ds
	-\frac{t}{k_n}\int_{t_{n-1}}^{t_n} u'(s)\,ds
\]
and thus
$\|u-\Lagint u\|_{I_n}\le 2 \int_{t_{n-1}}^{t_n} \|u'(s)\|\,ds
    \le CM t_n^\sigmapp\le CM(k_1+k_2)^\sigmapp\le CMk^{\gamma\sigmapp}$.
If $n\ge3$ and $t\in I_n$, then we can write the interpolation error
in terms of a divided difference,
$(u-\Lagint u)(t)=u[t_{n-2},t_{n-1},t,t_n](t-t_{n-2})(t-t_{n-1})(t-t_n)$, so
\[
\|u-\Lagint u\|_{I_n}\le\tfrac14k_n^2(k_{n-1}+k_n)\tfrac{1}{3!}
    \|u'''\|_{[t_{n-2},t_n]}
    \le CMk_n^2(k_{n-1}+k_n)t_n^{\sigmapp-3},
\]
where, in the final step, we used~\eqref{eq: kn}.
If $1\le\gamma<3/\sigmapp$ then, again using \eqref{eq: kn},
\[
k_n^2(k_{n-1}+k_n)t_n^{\sigmapp-3}
    \le C(kt_n^{1-1/\gamma})^{\gamma\sigmapp}k_n^{3-\gamma\sigmapp}
    t_n^{\sigmapp-3}= Ck^{\gamma\sigmapp}(k_n/t_n)^{3-\gamma\sigmapp}
    \le Ck^{\gamma\sigmapp},
\]
but for~$\gamma\ge3/\sigma$,
\[
k_n^2(k_{n-1}+k_n)t_n^{\sigmapp-3}
\le C(kt_n^{1-1/\gamma})^3t_n^{\sigmapp-3}
\le Ck^3t_n^{\sigmapp-3/\gamma}\le CT^{\sigmapp-3/\gamma}k^3.
\eqno\endproof
\]

Now consider the stability of the interpolation operator~$\Lagint$.
We see from~\eqref{eq: L I1} that
\[
\|\Lagint v\|_{I_1}\le\max\bigl(|v^0_-|,|v^1_-|\bigr)
\quad\text{and}\quad
\|\Lagint v\|_{I_2}\le\max\bigl(|v^1_-|,|v^2_-|\bigr).
\]
A similar estimate holds for the subsequent subintervals
provided the mesh satisfies the local quasi-uniformity condition
\begin{equation}\label{eq: Lambda}
k_n\le\Lambda k_{n-1}\quad\text{for $3\le n\le N$.}
\end{equation}
For example, our standard mesh~\eqref{eq: std tn}
satisfies this condition with~$\Lambda=2^\gamma-1$.

\begin{lemma}\label{lem: L stab}
If \eqref{eq: Lambda} holds, then
\[
\|\Lagint v\|_{I_n}\le\bigl(2+\tfrac54\Lambda\bigr)\max_{n-2\le j\le n}|v^j_-|
	\quad\text{for $3\le n\le N$.}
\]
\end{lemma}
\begin{proof}
The estimate follows from~\eqref{eq: L In} because, for $t\in I_n$ and $n\ge2$,
\begin{gather*}
\frac{|(t-t_{n-1})(t-t_n)|}{k_{n-1}(k_{n-1}+k_n)}
    \le\frac{\tfrac14k_n^2}{k_{n-1}(k_{n-1}+k_n)}
    \le\frac{\tfrac14k_n}{k_{n-1}}\le\tfrac14\Lambda,\\
\frac{|(t-t_{n-2})(t-t_n)|}{k_{n-1}k_n}
    \le\frac{(k_{n-1}+k_n)k_n}{k_{n-1}k_n}
    =1+\frac{k_n}{k_{n-1}}\le 1+\Lambda,\\
\frac{|(t-t_{n-2})(t-t_{n-1})|}{(k_{n-1}+k_n)k_n}
    \le\frac{(k_{n-1}+k_n)k_n}{(k_{n-1}+k_n)k_n}=1.
\end{gather*}
\end{proof}

Hence, the interpolant~$\Upp$ is superconvergent, uniformly in~$t$.

\begin{theorem}\label{thm: PP U-u}
Suppose that the time mesh satisfies \eqref{eq: kn}, \eqref{eq: k1}~and
\eqref{eq: Lambda}, and that $u$ satisfies \eqref{eq: u'''}.
If $\gamma\ge3/\sigmapp$, then the postprocessed solution~\eqref{eq: U*}
satisfies
\[
\|\Upp-u\|_J\le\max_{0\le n\le N}\|U^n_--u(t_n)\|
	+C_T\Lambda M\times\begin{cases}
		k^{\gamma\sigmapp},&1\le\gamma<3/\sigmapp,\\
		k^3,&\gamma\ge3/\sigmapp.
\end{cases}
\]
\end{theorem}
\begin{proof}
Write $\Upp-u=(\Lagint u-u)+\Lagint(U-u)$
and apply Lemmas \ref{lem: L error}~and \ref{lem: L stab}.
\end{proof}
%
\section{Spatial discretization}\label{sec: space}
%
\subsection{The fully discrete DG method}\label{subsec: FD DG}
We denote the norm of~$u$ in~$H^r(\Omega)$ by~$\|u\|_r$, and
assume now that $A=-\nabla^2$ in a bounded, convex or $C^2$ 
domain~$\Omega$ in~$\R^d$, subject to homogeneous Dirichlet boundary 
conditions.  Thus, if $u\in H^1_0(\Omega)$ and $Au\in L_2(\Omega)$, 
then $u\in H^2(\Omega)$~and $\|u\|_2\le C\|Au\|_{L_2(\Omega)}$. Let
$S_h\subseteq D(A^{1/2})=H^1_0(\Omega)$ denote the space of 
continuous, piecewise-linear functions with respect to a
quasi-uniform partition of~$\Omega$ into triangular or quadrilateral
(or tetrahedral etc.)
finite elements, with maximum diameter~$h$.  Recall that the
$L_2$-projector $P_h:L_2(\Omega)\to S_h$ and the Ritz
projector~$R_h:H^1_0(\Omega)\to S_h$ are defined by
\begin{equation}\label{eq: Rh def}
\iprod{P_hv,W}=\iprod{v,W}
\quad\text{and}\quad
A(R_hv,W)=A(v,W)\quad\text{for all $W\in S_h$,}
\end{equation}
and that the latter has the quasi-optimal approximation property
\begin{equation}\label{eq: Rh error}
\|v-R_hv\|+h\|\nabla(v-R_hv)\|\le Ch^2\|v\|_2
    \quad\text{for $v\in H^1_0(\Omega)\cap H^2(\Omega)$.}
\end{equation}
Let $\Trial(S_h)$ denote
the space of piecewise linear functions~$U:J\to S_h$ (so $U$ is continuous
in space, but may be discontinuous in time).

We define the fully discrete DG solution~$U_h\in\Trial(S_h)$ by requiring
\eqref{eq: DG step} to hold for every~$X\in\Trial(S_h)$.  Equivalently,
cf.~\eqref{eq: GN U},
\begin{equation}\label{eq: FD}
G_N(U_h,X)=\iprod{U_{h-}^0,X^0_+}+\int_0^{t_N}\iprod{f(t),X(t)}\,dt
    \quad\text{for all $X\in\Trial(S_h)$,}
\end{equation}
where, for simplicity, we choose $U_{h-}^0=(U_h)^0_-=P_hu_0$.  In view
of~\eqref{eq: GN(u,X)}, the Galerkin orthogonality
property~\eqref{eq: Gal orthog} now takes the form
\begin{equation}\label{eq: Gal orthog FD}
G_N(U_h-u,X)=\iprod{P_hu_0-u_0,X^0_+}=0\quad\text{for all $X\in\Trial(S_h)$.}
\end{equation}
Similarly, the fully discrete DG solution~$Z_h\in\Trial(S_h)$ for the
dual problem~\eqref{eq: dual} is defined by
\begin{equation}\label{eq: dual FD}
G_N(V,Z_h)=\iprod{V^N_-,Z_{h+}^N}\quad\text{for all $V\in\Trial(S_h)$,}
\quad\text{with $Z_{h+}^N=P_hz_T$,}
\end{equation}
and, since $z$ satisfies~\eqref{eq: dual weak},
\begin{equation}\label{eq: dual Gal orthog FD}
G_N(V,Z_h-z)=\iprod{V^N_-,P_hz_T-z_T}=0\quad\text{for all $V\in\Trial(S_h)$.}
\end{equation}
Theorem~\ref{thm: nodal error} generalizes as follows.

\begin{theorem}\label{thm: nodal error FD}
If $u$~and $z$ are the solutions of the initial value problem~\eqref{eq: ivp}
and the dual problem~\eqref{eq: dual}, and if $U_h$~and $Z_h$ are the
corresponding fully discrete DG solutions satisfying
\eqref{eq: FD}~and \eqref{eq: dual FD}, then
\[
\bigiprod{U_{h-}^N-u(t_N),z_T}=G_N(u-\Pi^-R_hu,Z_h-z)
    \quad\text{for every $z_T\in\Hilb$.}
\]
\end{theorem}
\begin{proof}
By taking $V=U_h$ in~\eqref{eq: dual FD} we find that
\[
\iprod{U_{h-}^N,z_T}=\iprod{U_{h-}^N,P_hz_T}
    =\iprod{U_{h-}^N,Z_{h+}^N}=G_N(U_h,Z_h),
\]
and taking $v=u$ in~\eqref{eq: dual weak} we have
$\iprod{u(t_N),z_T}=G_N(u,z)$, so
\begin{align*}
\bigiprod{U_{h-}^N-u(t_N),z_T}&=G_N(U_h,Z_h)-G_N(u,z)\\
    &=G_N(U_h-u,Z_h)+G_N(u,Z_h-z)=G_N(u,Z_h-z),
\end{align*}
where the final step used \eqref{eq: Gal orthog FD} with~$X=Z_h$.
Since
\[
G_N(u,Z_h-z)=G_N(u-\Pi^-R_hu,Z_h-z)+G_N(\Pi^-R_hu,Z_h-z),
\]
the result follows after putting $V=\Pi^-R_hu$
in~\eqref{eq: dual Gal orthog FD}.
\end{proof}

\subsection{Error in the fully discrete DG solution of the dual problem}
We modify the splitting~\eqref{eq: split Z-z}, by writing
$A^{-1}(Z_h-z)=\zeta+\Psi+\Phi$ where
\[
\zeta=A^{-1}(\Pi^+z-z),\quad
\Psi=A^{-1}\Pi^+(P_hz-z),\quad
\Phi=A^{-1}(Z_h-\Pi^+P_hz)\in\Trial(S_h).
\]
Theorem~\ref{thm: weak ||Z-z||} generalizes as follows.

\begin{theorem}\label{thm: Ainv(Z-z)}
Let $z$ denote the solution of the dual problem~\eqref{eq: dual}, and
let $Z_h$ denote the fully discrete DG solution defined
by~\eqref{eq: dual FD}.  Then, for $-1<\alpha<1$,
\[
\|A^{-1}(Z_h-z)\|_J\le C\bigl(t_N^{\alpha_+}k^{1+\alpha_-}+h^2\bigr)
    \ell(t_N/k_N)\|z_T\|.
\]
\end{theorem}
\begin{proof}
We already estimated~$\|\zeta\|$ in Lemma~\ref{lem: zeta}.
To estimate~$\Psi$, observe that since $A^{-1}$ commutes with~$\Pi^+$
and since $AR_h=P_hA$ (implying $A^{-1}P_h=R_hA^{-1}$),
\begin{equation}\label{eq: Psi}
\Psi=\Pi^+(R_h-\Id)A^{-1}z.
\end{equation}
Using \eqref{eq: Pi norm}, the error bound~\eqref{eq: Rh error} for
the Ritz projection, $H^2$-regularity for~$A$ and
Lemma~\ref{lem: dual reg}, we find that
\begin{equation}\label{eq: ||Psi||}
\|\Psi\|_{I_n}\le3\|(R_h-\Id)A^{-1}z\|_{I_n}
    \le Ch^2\|A^{-1}z(t)\|_2\le Ch^2\|z(t)\|\le Ch^2\|z_T\|.
\end{equation}

To estimate~$\Phi$, observe that since $A^{-1}V=A^{-1}P_hV=R_hA^{-1}V$,
\begin{equation}\label{eq: GN(V,Phi)}
G_N(V,\zeta+\Psi+\Phi)=G_N(A^{-1}V,Z_h-z)=G_N(R_hA^{-1}V,Z_h-z)=0,
\end{equation}
where we used \eqref{eq: dual Gal orthog FD} with~$V$
replaced by~$R_hA^{-1}V$.  From the proof of Lemma~\ref{lem: Theta},
\begin{equation}\label{eq: GN(V,zeta)}
G_N(V,\zeta)=\int_0^T\iprod{V,\B_\alpha^*A\zeta}\,dt,
\end{equation}
and by~\eqref{eq: GN dual},
\begin{equation}\label{eq: GN(V,Psi)}
G_N(V,\Psi)=\iprod{V^N_-,\Psi^N_-}+\sum_{n=1}^{N-1}\iprod{V^n_-,[\Psi]^n}
    +\sum_{n=1}^N\int_{I_n}\bigl[-\iprod{V,\Psi'}+A(V,\B_\alpha^*\Psi)
    \bigr]\,dt.
\end{equation}
Since $A(V,\B^*_\alpha\Psi)=A\bigl(V,\B_\alpha^*A^{-1}\Pi^+(P_h-\Id)z)
=\iprod{V,(P_h-\Id)\B_\alpha^*\Pi^+ z}=0$,
\[
|G_N(V,\Psi)|\le\|V\|_J\biggl(\|\Psi^N_-\|+\sum_{n=1}^{N-1}\|[\Psi]^n\|
    +\sum_{n=1}^N\int_{I_n}\|\Psi'\|\,dt\biggr).
\]
By~\eqref{eq: ||Psi||}, $\|\Psi^N_-\|\le\|\Psi\|_{I_N}\le Ch^2\|z_T\|$.
Using \eqref{eq: Psi}, \eqref{eq: Pi norm}, \eqref{eq: Rh error}~and
Lemma~\ref{lem: dual reg}, we have
$\|[\Psi]^n\|\le\int_{I_n}\|(R_h-\Id)A^{-1}z'(t)\|\,dt
\le Ch^2\|z_T\|\int_{I_n}(T-t)^{-1}\,dt$ so
\[
\sum_{n=1}^{N-1}\|[\Psi]^n\|\le Ch^2\|z_T\|\int_0^{t_{N-1}}(T-t)^{-1}\,dt
    = Ch^2\|z_T\|\log(t_N/k_N).
\]
Using \eqref{eq: Pi norm}, \eqref{eq: Psi} and Lemma~\ref{lem: dual reg},
we find that
\[
\int_{I_n}\|\Psi'(t)\|\,dt\le\frac{2}{k_n}
    \int_{I_n}(t_n-t)\|(R_h-\Id)A^{-1}z'(t)\|\,dt
    \le\frac{Ch^2}{k_n}\,\|z_T\|\int_{I_n}\frac{t_n-t}{T-t}\,dt,
\]
so
\[
\sum_{n=1}^N\int_{I_n}\|\Psi'\|\,dt
    \le Ch^2\|z_T\|\biggl(\int_0^{t_{N-1}}(T-t)^{-1}\,dt
    +k_N^{-1}\int_{I_N}\,dt\biggr),
\]
and we conclude that $|G_N(V,\Psi)|\le Ch^2\ell(t_N/k_N)\|V\|_J\|z_T\|$.
Therefore, by \eqref{eq: GN(V,Phi)}, \eqref{eq: GN(V,zeta)}
and Lemma~\ref{lem: bilinear},
\begin{equation}\label{eq: |GN|}
|G_N(V,\Phi)|=|G_N(V,\zeta)+G_N(V,\Psi)|
    \le C\bigl(t_N^{\alpha_+}k^{1+\alpha_-}+h^2\bigr)
        \ell(t_N/k_N)\|V\|_J\|z_T\|.
\end{equation}

Fix~$n$ with~$1\le n\le N$, and define $V\in\Trial(S_h)$ by
\[
V(t)=\begin{cases}
    0,&\text{if $t\in I_j$ for $1\le j\le n-1$,}\\
    \Phi(t),&\text{if $t\in I_j$ for $n\le j\le N$.}
\end{cases}
\]
In view of~\eqref{eq: GN(V,V)},
$G_N(V,\Phi)\ge\tfrac12\|\Phi^N_-\|^2+\tfrac12\|\Phi^{n-1}_+\|^2
    +\tfrac12\sum_{j=n}^{N-1}\|[\Phi]^j\|^2$,
so the estimate~\eqref{eq: |GN|} gives
$\|\Phi^{n-1}_+\|^2+\|[\Phi]^n\|^2
\le C\bigl(t_N^{\alpha_+}k^{1+\alpha_-}+h^2\bigr)\ell(t_N/k_N)
\|\Phi\|_{(t_{n-1},T)}\|z_T\|$ for~$1\le n\le N-1$, whereas
\[
\|\Phi^{N-1}_+\|^2+\|\Phi^N_-\|^2\le C\bigl(
t_N^{\alpha_+}k^{1+\alpha_-}+h^2\bigr)\ell(t_N/k_N)
    \|\Phi\|_{(t_{N-1},T)}\|z_T\|.
\]
Furthermore, $\|\Phi\|_{I_n}=\max\bigl(\|\Phi^{n-1}_+\|,\|\Phi^n_-\|\bigr)$
because $\Phi$ is piecewise linear in~$t$, and
$\|\Phi^n_-\|\le\|\Phi^n_+\|+\|[\Phi]^n\|$, implying that
$\|\Phi\|_{I_n}^2\le\|\Phi^{n-1}_+\|^2+\|\Phi^n_+\|^2+\|[\Phi]^n\|^2$.
By letting $n^*=\argmax_{1\le n\le N}\|\Phi\|_{I_n}$, we see that
\[
\|\Phi\|_J^2=\|\Phi\|_{I_{n^*}}^2
    \le C\bigl(t_N^{\alpha_+}k^{1+\alpha_-}+h^2\bigr)\ell(t_N/k_N)
    \|\Phi\|_J\|z_T\|,
\]
giving the desired bound for~$\|\Phi\|_J$.
\end{proof}
\subsection{Fully-discrete nodal error}
As claimed in the Introduction, we have the following error bound 
for~$U_h$.	

\begin{theorem}\label{thm: Uh- nodal}
Assume that the solution~$u$ of the initial value problem~\eqref{eq: ivp} 
satisfies \eqref{eq: Au'}~and \eqref{eq: A^2u'}, and that
the time mesh satisfies assumptions \eqref{eq: kn}~and \eqref{eq: k1} 
with~$\gamma>\gamma^*=(2+\alpha_-)/\sigma$.  Then, 
the fully discrete DG solution~$U_h\in\Trial(S_h)$ satisfies
\[
\|U_{h-}^n-u(t_n)\|\le C_TM\bigl(k^{3+2\alpha_-}\ell(t_n/k_n)+h^2\bigr)
    \quad\text{for $0\le n\le N$.}
\]
\end{theorem}
\begin{proof}
In view of Lemma~\ref{lem: nodal conv}, it suffices to show
(cf.~Theorem~\ref{thm: ||nodal error||}) that
\[
\|U_{h-}^n-u(t_n)\|\le C_T\bigl(k^{1+\alpha_-}
    +h^2\bigr)\,\ell(t_n/k_n)\epsilon(u)
    +C_TMh^2.
\]
Put $\xi=u-R_h u$ and $\eta=u-\Pi^-u$ so that
$u-\Pi^-R_hu=\eta+\Pi^-\xi$ and thus, by
Theorem~\ref{thm: nodal error FD}, $\iprod{U_{h-}^N-u(t_N),z_T}
=G_N(\eta,Z_h-z)+G_N(\Pi^-\xi,Z_h-z)$.
Using Theorem~\ref{thm: Ainv(Z-z)} in place of
Theorem~\ref{thm: weak ||Z-z||}, we can show
as in the proof of Theorem~\ref{thm: ||nodal error||} that
$|G_N(\eta,Z_h-z)|\le C_T\|z_T\|\bigl(k^{1+\alpha_-}+h^2\bigr)
\logp(t_n/k_n)\,\epsilon(u)$.
By~\eqref{eq: GN def}, $G_N(\Pi^-\xi,Z_h)$ equals
\[
\bigiprod{(\Pi^-\xi)^0_+,Z_{h+}^0}+\sum_{n=1}^{N-1}
    \bigiprod{[\Pi^-\xi]^n,Z_{h+}^n}
    +\sum_{n=1}^N\int_{I_n}\bigl[\bigiprod{(\Pi^-\xi)',Z_h}
        +A(\B_\alpha\Pi^-\xi,Z_h)\bigr]\,dt,
\]
and, since $\B_\alpha$ commutes with the Ritz projector~$R_h$, the
definition~\eqref{eq: Rh def} of~$R_h$ implies that
$A(\B_\alpha\Pi^-\xi,Z_h)=A(\B_\alpha\Pi^-u,Z_h)
-A(R_h\B_\alpha\Pi^-u,Z_h)=0$.
Integrating by parts, applying the interpolation and orthogonality
properties~\eqref{eq: Pi properties} of~$\Pi^-$, and noting that
$\xi^n_-=\xi(t_n)=(\Pi^-\xi)^n_-$ and that $Z_h'$ is constant on~$I_n$,
\begin{align*}
\int_{I_n}\bigiprod{(\Pi^-\xi)',Z_h}\,dt&=\bigiprod{(\Pi^-\xi)^n_-,Z_{h-}^n}
    -\bigiprod{(\Pi^-\xi)^{n-1}_+,Z^{n-1}_{h+}}
    -\int_{I_n}\bigiprod{\Pi^-\xi,Z_h'}\,dt\\
    &=\bigiprod{\xi^n_-,Z^n_{h-}}
        -\bigiprod{(\Pi^-\xi)^{n-1}_+,Z^{n-1}_{h+}}
        -\int_{I_n}\iprod{\xi,Z_h'}\,dt\\
    &=\bigiprod{\xi^{n-1}_+-(\Pi^-\xi)^{n-1}_+,Z^{n-1}_{h+}}
        +\int_{I_n}\iprod{\xi',Z_h}\,dt
\end{align*}
so $G_N(\Pi^-\xi,Z_h)=\bigiprod{\xi(0),Z^0_{h+}}
+\sum_{n=1}^{N-1}\bigiprod{[\xi]^n,Z^n_{h+}}
+\sum_{n=1}^N\int_{I_n}\iprod{\xi',Z_h}\,dt$.
Using \eqref{eq: dual weak} with~$v=\Pi^-\xi$, and noting that $[\xi]^n=0$,
we obtain
\begin{align*}
G_N(\Pi^-\xi,Z_h-z)&=\bigiprod{\xi(0),Z^0_{h+}}-\bigiprod{\xi(T),z_T}
    +\sum_{n=1}^N\int_{I_n}\iprod{\xi',Z_h}\,dt.
\end{align*}
Stability of the fully discrete dual problem, $\|Z_h\|_J\le C\|z_T\|$, 
follows from~\eqref{eq: GN(V,V)}, so
\begin{align*}
\bigl|G_N(\Pi^-\xi,Z_h-z)\bigr|
    &\le C\|z_T\|\biggl(\|\xi(0)\|+\|\xi(t_N)\|
    +\int_0^T\|\xi'\|\,dt\biggr)\\
    &\le Ch^2\|z_T\|\biggl(\|Au_0\|+\|Au(T)\|
        +\int_0^T\|Au'(t)\|\,dt\biggr),
\end{align*}
where we used the error bound~\eqref{eq: Rh error} for the Ritz projector.
The result follows using the regularity assumption~\eqref{eq: Au'}.
\end{proof}

\subsection{Postprocessing the fully discrete DG solution}\label{sec: Upp}
Theorem~\ref{thm: PP U-u} remains valid if $\Upp=\Lagint U$~and $U^n_-$
are replaced by $\Upp_h=\Lagint U_h$~and $U^n_{h-}$, respectively.
%
\section{Numerical results}\label{sec: Numerical}
%
We present a series of numerical tests using a model problem in one space dimension,
of the form~\eqref{eq: ivp} with
\[
Au=-u_{xx},\quad\Omega=(0,1),\quad[0,T]=[0,1],\quad u_0(x)=x(1-x),\quad f\equiv0,
\]
and homogeneous Dirichlet (absorbing) boundary conditions.  These tests reveal 
faster than expected convergence when~$\alpha<0$, and that our regularity 
assumptions are more restrictive than is needed in practice.
We apply the fully discrete DG method defined in Section~\ref{subsec: FD DG},
employing a time mesh of the form~\eqref{eq: std tn}, for various choices of 
the mesh grading parameter~$\gamma\ge1$, and a uniform spatial mesh consisting of 
$M$~subintervals, each of length $h=1/M$.  We always choose 
$M=\lceil N^{3/2}\rceil$ so that $h^2\approx k^3$ and hence the error from 
the time discretization dominates the spatial error.

\subsection{The exact solution}\label{subsec: exact}
Separation of variables yields a series representation 
\begin{equation}\label{eq: reg0}
u(x,t)=8\sum_{n=0}^\infty \omega_n^{-3} \sin(\omega_n x) 
	E_{1+\alpha}(-\omega_n^2 t^{1+\alpha})\quad
	\text{with $\omega_n=(2n+1)\pi$,}
\end{equation}
where the Mittag--Leffler function is given by
$E_\nu( t)=\sum_{p=0}^\infty t^p/\Gamma(1+\nu p)$.  We can verify directly
that $u$ satisfies the regularity conditions
\begin{equation}\label{eq: u' practice}
t^{1+\alpha}\|Au'(t)\|+t^{2+\alpha}\|A u''(t)\|	
	\le Mt^{\sigma-1}\quad\text{for $0<t\le T$,}
\end{equation}
with
\begin{equation}\label{eq: ||u||_2 practice}
\|u(t)\|_2+t\|u'(t)\|_2\le M\quad\text{for $0<t\le T$.}
\end{equation}
In fact, by differentiating~\eqref{eq: reg0}, 
\[
\partial_t^ju_{xx}(x,t)=-8\sum_{n=0}^\infty \omega_n^{-1}\sin(\omega_n x) 
		\frac{d^j}{dt^j}E_{1+\alpha} (-\omega_n^2 t^{1+\alpha})
	\quad\text{for $j\in\{1,2\}$,}
\]
so by Parseval's identity,
\[
\|\partial_t^jAu(t)\|^2=\|\partial_t^ju_{xx}(t)\|^2
	=32\sum_{n=0}^\infty \omega_n^{-2}
	\biggl(\frac{d^j}{dt^j}E_{1+\alpha}(-\omega_n^2 t^{1+\alpha})
	\biggr)^2. 
\]
The Mittag--Leffler function satisfies~\cite[Theorem~4.2]{McLean2010}
\begin{equation}\label{eq: bound ml}
\biggl|\frac{d^j}{dt^j}E_{1+\alpha}(-\omega_n^2 t^{1+\alpha})\biggr|
	\le C t^{-(1+\alpha)\mu -j} \omega_n^{-2\mu}\quad
	\text{for $j\in\{1,2,3,\ldots\}$ and $|\mu|\le 1$,}
\end{equation} 
and taking $\mu =-\epsilon$ yields
\[
\bigl(t^{j+\alpha}\|\partial_t^jAu(t)\|\bigr)^2
	\le C t^{2\epsilon(1+\alpha)+2\alpha}\sum_{n=0}^\infty
	\omega_n^{4\epsilon-2}
	\le\frac{C\bigl(t^{\epsilon(1+\alpha)+\alpha}\bigr)^2}{1-4\epsilon}
	\quad\text{for $-1<\epsilon<\tfrac14$.}
\]
Thus, the regularity condition~\eqref{eq: u' practice} holds 
for~$\sigma=(1+\epsilon)(1+\alpha)<\tfrac54(1+\alpha)$~and 
$M=C(\tfrac14-\epsilon)^{-1/2}$.  In particular, putting $\epsilon=0$ gives
the bound for~$t\|u'(t)\|_2$ in~\eqref{eq: ||u||_2 practice}, and since
$|E_{1+\alpha}(-\omega_n^2t^{1+\alpha})|\le C$ for all~$t>0$ we also
have $\|u(t)\|_2^2\le C\sum_{n=0}^\infty\omega_n^{-2}<\infty$.
However, $u$ fails to satisfy the second regularity
assumption~\eqref{eq: A^2u'} used in our theoretical analysis.

\begin{table}
\renewcommand{\arraystretch}{1}
\begin{center}
\caption{The left nodal error $\max_{1\le n\le N}\,||U_{h-}^n-u(t_n)||$ 
and the rate of convergence when~$\alpha=-0.3$, for different mesh gradings~$\gamma$.}
\label{tab: left nodal alpha neg}
\begin{tabular}{|r|rr|rr|rr|rr|rr|}
\hline {$N$}&\multicolumn{2}{c|}{$\gamma=1$} &\multicolumn{2}{c|}{$\gamma=2$}
&\multicolumn{2}{c|}{$\gamma=3$}&\multicolumn{2}{c|}{$\gamma=3.25$}\\
\hline
 20& 2.01e-03&      & 1.08e-04&      & 6.39e-05&      & 6.39e-05&      \\
 40& 8.61e-04& 1.220& 3.15e-05& 1.780& 1.09e-05& 2.546& 1.10e-05& 2.535\\
 80& 3.90e-04& 1.143& 9.33e-06& 1.758& 1.81e-06& 2.596& 1.82e-06& 2.595\\
160& 2.21e-04& 0.821& 2.77e-06& 1.753& 2.92e-07& 2.632& 2.94e-07& 2.632\\
\hline
\end{tabular}
\vspace{2ex}
\caption{The right nodal error $\max_{0\le n\le N-1}\,||U_{h+}^n-u(t_n)||$ 
and the rate of convergence when~$\alpha=-0.3$, for different mesh gradings~$\gamma$.}
\label{tab: right nodal alpha neg}
\begin{tabular}{|r|rr|rr|rr|rr|rr|}
\hline {$N$}&\multicolumn{2}{c|}{$\gamma=1$} &\multicolumn{2}{c|}{$\gamma=2$}
&\multicolumn{2}{c|}{$\gamma=3$}&\multicolumn{2}{c|}{$\gamma=3.25$}\\
\hline
 20& 4.74e-02&      & 6.03e-03&      & 1.63e-03&      & 1.52e-03&      \\
 40& 3.05e-02& 0.636& 2.26e-03& 1.416& 4.18e-04& 1.966& 3.91e-04& 1.964\\
 80& 1.89e-02& 0.689& 8.51e-04& 1.410& 1.06e-04& 1.982& 9.89e-05& 1.982\\
160& 1.16e-02& 0.710& 3.21e-04& 1.406& 2.66e-05& 1.990& 2.49e-05& 1.989\\
\hline
\end{tabular}
\end{center}
\end{table}
\subsection{Nodal errors}
The numerical results described below suggest that 
\begin{equation}\label{eq: observed nodal error}
\underset{1\le n\le N}{\max}\|U^n_{h-}-u(t_n)\|\le Ch^2+C\times\begin{cases}
        k^{\gamma\sigma},&1\le\gamma\le(3+\alpha_-)/\sigma,\\
        k^{3+\alpha_-},&\gamma>(3+\alpha_-)/\sigma.
        \end{cases}
\end{equation}
Thus, the time discretization error appears to be $O(k^{3+\alpha_-})$
for $\gamma>(3+\alpha_-)/\sigma$, compared to our theoretical bound
of~$O(k^{3+2\alpha_-})$ for~$\gamma>(2+\alpha_-)/\sigma$, where the latter
assumes the stronger regularity conditions 
\eqref{eq: Au'}~and \eqref{eq: A^2u'}.  

For~$\alpha=-0.3$, we observe in Table~\ref{tab: left nodal alpha neg} 
convergence of order~$k^{1.25 \gamma(\alpha+1)}$ 
for~$1\le\gamma \le (3+\alpha)/[1.25(\alpha+1)]\approx 3.086$.  
In particular, the highest observed convergence rate is $O(k^{3+\alpha_-})$, 
and not $O(k^{3+2\alpha_-})$ as expected from Theorem~\ref{thm: Uh- nodal}.
Table~\ref{tab: right nodal alpha neg} shows that the right-hand 
limit~$U^n_{h+}=U_h(t_n^+)=\lim_{t\to t_n^+}U_h(t)$
is not a superconvergent approximation to~$u(t_n)$; the error
is $O(k^2)$ at best.

For~$\alpha=+0.3$, Table~\ref{tab: left nodal alpha pos} shows
convergence of order~$k^{1.25\gamma(\alpha+1)}$ 
for~$1\le\gamma \le 3/[1.25(\alpha+1)]\approx 1.85$, so in the best case
the error is~$O(k^3)$, consistent with Theorem~\ref{thm: Uh- nodal}.
In Table~\ref{tab: right nodal alpha pos}, we see that $U_h^+$ again
fails to be superconvergent.

\begin{table}
\renewcommand{\arraystretch}{1}
\begin{center}
\caption{The left nodal error $\max_{1\le n\le N}\,||U_{h-}^n-u(t_n)||$ 
and the rate of convergence when~$\alpha=+0.3$, for different mesh gradings~$\gamma$.}
\label{tab: left nodal alpha pos}
\begin{tabular}{|r|rr|rr|rr|rr|rr|}
\hline {$N$}&\multicolumn{2}{c|}{$\gamma=1$} &\multicolumn{2}{c|}{$\gamma=1.5$}
&\multicolumn{2}{c|}{$\gamma=1.75$}&\multicolumn{2}{c|}{$\gamma=2$}\\
\hline
 20& 2.10e-04&      & 2.08e-05&      & 1.21e-05&      & 1.23e-05&      \\
 40& 6.77e-05& 1.632& 3.61e-06& 2.527& 1.61e-06& 2.904& 1.57e-06& 2.966\\
 80& 2.19e-05& 1.636& 6.43e-07& 2.486& 2.13e-07& 2.917& 1.99e-07& 2.983\\
160& 7.11e-06& 1.625& 1.17e-07& 2.461& 2.80e-08& 2.930& 2.53e-08& 2.972\\
\hline
\end{tabular}
\vspace{2ex}
\caption{The right nodal error $\max_{0\le n\le N-1}\,||U_{h+}^n-u(t_n)||$ 
and the rate of convergence when~$\alpha=+0.3$, for different mesh gradings~$\gamma$.}
\label{tab: right nodal alpha pos}
\begin{tabular}{|r|rr|rr|rr|rr|}
\hline {$N$}&\multicolumn{2}{c|}{$\gamma=1$} &\multicolumn{2}{c|}{$\gamma=1.5$}
&\multicolumn{2}{c|}{$\gamma=1.75$}\\
\hline
 20& 3.265e-03&      & 8.548e-04&      & 9.207e-04&      \\
 40& 1.536e-03& 1.088& 2.165e-04& 1.982& 2.338e-04& 1.977\\
 80& 6.726e-04& 1.191& 5.432e-05& 1.995& 5.873e-05& 1.993\\
160& 2.851e-04& 1.238& 1.361e-05& 1.997& 1.472e-05& 1.996\\
\hline
\end{tabular}
\end{center}
\end{table}
Given $\alpha$, it is natural to ask which value of~$\gamma$ leads to the 
smallest error.  Figure~\ref{fig: error against gamma} shows the maximum 
nodal error (on a logarithmic scale) as a function of~$\gamma\in[1,8]$
for 4~choices of~$\alpha$, when $M=512$~and $N=64$ (so $h^2=k^3$).  
The error is minimised when~$\gamma\approx(3+\alpha_-)/\sigma$; 
for instance, in the case~$\alpha=0.2$ the best choice is 
$\gamma\approx3/[\tfrac54(1.2)]=2$.
In Figure~\ref{fig: hp error against alpha}, we instead show
the maximum nodal error as a function of~$\alpha\in[-0.9,0.9]$
for 4~choices of~$\gamma$.  The benefit from using non-uniform time 
steps is clear, except when~$\alpha$ is close to $-1$~or $1$.

\begin{figure}[htb]
\begin{center}
\caption{The left nodal error ${\rm max}_{0\le n\le N}\,||U_{h-}^n-u(t_n)||$ 
as a function of~$\gamma$, for~$\alpha=-0.8,-0.4,0.2$ and $0.6$, when 
$M=512$~and $N=64$ (so $h^2=k^3$).}\label{fig: error against gamma}
\scalebox{0.45}{\includegraphics{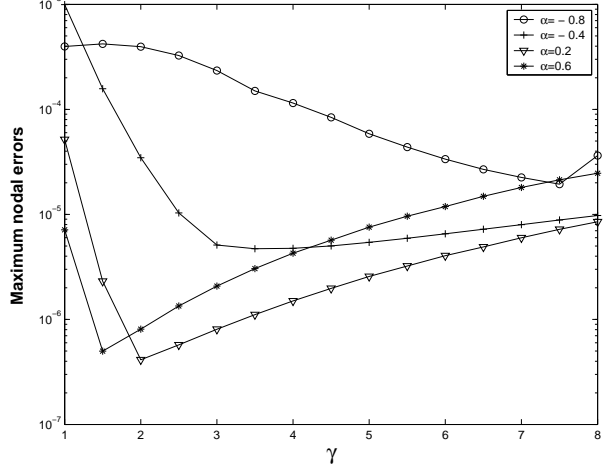}}
\end{center}
\end{figure}
\begin{figure}[htb]
\begin{center}
\caption{The left nodal error~${\rm max}_{0\le n\le N}\,||U_{h-}^n-u(t_n)||$ 
as a function of~$\alpha$, for $\gamma=1$, 2, 3, 4, when  $M=512$~and 
$N=64$ (so $k^3=h^2$).}\label{fig: hp error against alpha}
\scalebox{0.45}{\includegraphics{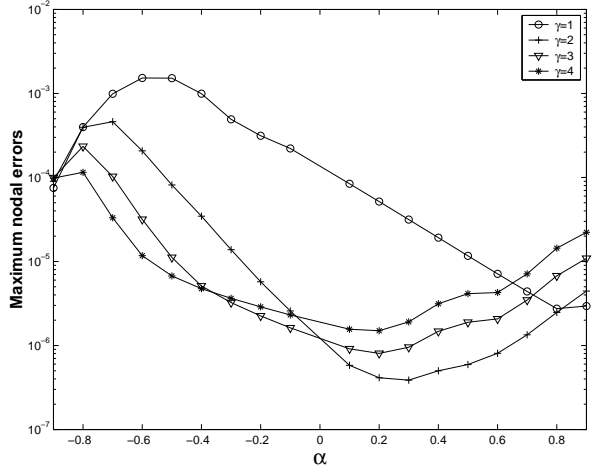}}
\end{center}
\end{figure}
\begin{table}
\renewcommand{\arraystretch}{1}
\begin{center}
\caption{The uniform DG error after postprocessing, $\|\Upp_h-u\|_{J,12}$, 
and its rate of convergence, when~$\alpha=-0.3$, for different mesh 
gradings~$\gamma$.}
\label{tab: pp dg alpha<0}
\begin{tabular}{|r|rr|rr|rr|rr|}
\hline 
$N$&\multicolumn{2}{c|}{$\gamma=1$} &\multicolumn{2}{c|}{$\gamma=2$}
&\multicolumn{2}{c|}{$\gamma=3$} &\multicolumn{2}{c|}{$\gamma=3.9$}\\
\hline
 20& 3.79e-02&      & 4.52e-03&      & 1.46e-03&      &8.13e-04& \\
 40& 2.37e-02& 0.675& 1.68e-03& 1.425& 3.27e-04& 2.154&1.20e-04& 2.763\\
 80& 1.44e-02& 0.716& 6.31e-04& 1.416& 7.49e-05& 2.127&1.79e-05& 2.743\\
160& 8.74e-03& 0.724& 2.38e-04& 1.410& 1.73e-05& 2.113&2.69e-06& 2.735\\
\hline
\end{tabular}
\vspace{2ex}
\caption{The uniform DG error after postprocessing, $\|\Upp_h-u\|_{J,12}$, 
and its rate of convergence, when~$\alpha=+0.3$, for different mesh 
gradings~$\gamma$.}
\label{tab: pp dg alpha>0}
\begin{tabular}{|r|rr|rr|rr|rr|rr|}
\hline 
$N$&\multicolumn{2}{c|}{$\gamma=1$} &\multicolumn{2}{c|}{$\gamma=1.5$}
&\multicolumn{2}{c|}{$\gamma=2$} &\multicolumn{2}{c|}{$\gamma=2.35$}\\
\hline
 20& 2.51e-03&      & 4.38e-04&      & 1.56e-04&      & 1.89e-04&       \\
 40& 1.16e-03& 1.120& 1.22e-04& 1.845& 2.72e-05& 2.515& 2.29e-05& 3.046\\
 80& 5.02e-04& 1.205& 3.34e-05& 1.867& 4.59e-06& 2.568& 2.80e-06& 3.029\\
160& 2.12e-04& 1.245& 8.88e-06& 1.911& 7.63e-07& 2.588& 3.44e-07& 3.024\\
\hline
\end{tabular}
\end{center}
\end{table}
\subsection{Global error after post-processing}
We introduce a finer mesh
\begin{equation}\label{eq: fine grid}
\G^{N,m}=\{\,t_{j-1}+\ell k_j/m:\text{$j=1$, 2, \dots, $N$ and
    $\ell=0$, 1, \dots, $m$}\,\},
\end{equation}
and define the discrete maximum 
norm~$\|v\|_{J,m}=\max_{t\in\G^{N,m}}\|v(t)\|$, so that, for 
sufficiently large values of~$m$, $\|U-u\|_{J,m}$ approximates the
global error~$\|U-u\|_J$.  Now, in addition to the regularity 
assumptions \eqref{eq: u' practice}~and \eqref{eq: ||u||_2 practice}, 
we require that $u$ satisfies~\eqref{eq: u'''}.  In fact, we see 
from \eqref{eq: reg0}~and \eqref{eq: bound ml} that, with~$\mu=-1$,
\[
\bigl(t^{j-1}\|\partial_t^ju(t)\|\bigr)^2\le C\bigl(t^{(1+\alpha)-1}\bigr)^2
	\sum_{n=0}^\infty\omega_n^{-2}
	\le C\bigl(t^{(1+\alpha)-1}\bigr)^2,
\]
so \eqref{eq: u'''} holds for~$\sigmapp=1+\alpha$.  
Using Theorem~\ref{thm: PP U-u} (cf.~Subsection~\ref{sec: Upp}) and
\eqref{eq: observed nodal error} 
with~$\sigmapp=(1+\alpha)<\sigma\approx\tfrac54(1+\alpha)$, we expect
\[
\|\Upp_h-u\|_J\le Ch^2+C\times\begin{cases}
	k^{\gamma\sigmapp},&1\le\gamma\le(3+\alpha_-)/\sigmapp,\\
	k^{3+\alpha_-},&\gamma>(3+\alpha_-)/\sigmapp.
	\end{cases}
\]
We observe this convergence behaviour in Tables \ref{tab: pp dg alpha<0}~and 
\ref{tab: pp dg alpha>0}.
\section{Concluding remarks}
We have analysed a piecewise-linear DG method for the time discretization
of~\eqref{eq: ivp} --- a fractional diffusion ($-1<\alpha<0$) or wave 
($0<\alpha<1$) equation --- and proved superconvergence at the nodes, 
generalizing a known result for the classical heat equation.  Numerical 
experiments indicate that our theoretical error bounds are sharp 
if~$\alpha>0$, but not if~$\alpha<0$.  For generic regular data 
$u_0$~and $f$, derivatives of the exact solution are singular as~$t\to0$, 
but nevertheless by employing non-uniform time steps we achieve
a high convergence rate of~$O(k^{3+\alpha_-})$.  After postprocessing 
the solution, the same high accuracy is achieved for all~$t$, not just 
at the nodes.  We have also proved that the additional error arising 
from a spatial discretization by continuous piecewise-linear finite 
elements is essentially~$O(h^2)$.  In future work, we aim to treat the 
case when the initial data~$u_0$ is not smooth.
\bibliographystyle{siam}
\bibliography{superconv-refs}
\end{document}